\documentclass[reqno]{amsart}
\usepackage{booktabs}
\usepackage{pdflscape}
\usepackage{amssymb, latexsym, amsthm, enumitem, color, amsmath}

\usepackage{array}
\usepackage{blindtext}
\usepackage{longtable}
\usepackage{enumitem}
\usepackage{multirow}
\usepackage{multicol}
\usepackage{centernot}
\usepackage{mathtools}
\usepackage{bbm}
\usepackage{nicefrac}
\usepackage{upquote}

\usepackage[unicode=true,pdfusetitle,bookmarks=true,bookmarksnumbered=false,
bookmarksopen=false,breaklinks=true,pdfborder={0 0 0},
pdfborderstyle={},backref=false,colorlinks=true]{hyperref}
\definecolor{myblue}{rgb}{0.09,0.32,0.44} 
\hypersetup{pdfborder={0 0 0},pdfborderstyle={},colorlinks=true,linkcolor=myblue,citecolor=myblue,urlcolor=blue}

\newtheorem{thm}{Theorem}[section] 
\newtheorem*{thm*}{Theorem}

\newtheorem{cor}[thm]{Corollary}
\newtheorem{defn}[thm]{Definition}

\newtheorem{exmpl}[thm]{Example}
\newtheorem{fact}[thm]{Fact}
\newtheorem{lem}[thm]{Lemma}
\newtheorem{prop}[thm]{Proposition}

\theoremstyle{remark}
\newtheorem{rem}[thm]{Remark}
\newtheorem{ques}[thm]{Question}
\newtheorem*{rem*}{Remark}
\newtheorem*{rems*}{Remarks}

\newcommand\Cref[1]{{Corollary~\ref{#1}}}
\newcommand\Dref[1]{{Definition~\ref{#1}}}
\newcommand\Fref[1]{{Fact~\ref{#1}}}
\newcommand\Lref[1]{{Lemma~\ref{#1}}}
\newcommand\Pref[1]{{Proposition~\ref{#1}}}
\newcommand\Rref[1]{{Remark~\ref{#1}}}
\newcommand\Tref[1]{{Theorem~\ref{#1}}}
\newcommand\Sref[1]{{\S \ref{#1}}}
\newcommand\Ssref[1]{{Subsection~\ref{#1}}}
\newcommand\Qref[1]{{Question~\ref{#1}}}

\newcommand{\N}{\mathbb{N}}
\newcommand{\Z}{\mathbb{Z}}

\newcommand{\F}{\mathrm{F}}
\newcommand{\set}[1]{\left\{#1\right\}}
\newcommand{\sub}{\subseteq}
\newcommand{\E}{\mathbb{E}}
\newcommand{\eps}{\varepsilon}

\newcommand{\ceil}[1]{\left\lceil #1 \right\rceil}

\newcommand{\sg}[1]{\left\langle #1\right\rangle}
\newcommand{\suchthat}{:}

\newcommand{\id}{\mathrm{id}}
\renewcommand{\Pr}{\mathbb{P}}
\newcommand{\supp}{\operatorname{supp}}
\newcommand{\on}{\textrm{ on }}

\newcommand{\Sym}{\mathrm{Sym}}

\makeatletter
\def\moverlay{\mathpalette\mov@rlay}
\def\mov@rlay#1#2{\leavevmode\vtop{%
   \baselineskip\z@skip \lineskiplimit-\maxdimen
   \ialign{\hfil$\m@th#1##$\hfil\cr#2\crcr}}}
\newcommand{\charfusion}[3][\mathord]{
    #1{\ifx#1\mathop\vphantom{#2}\fi
        \mathpalette\mov@rlay{#2\cr#3}
      }
    \ifx#1\mathop\expandafter\displaylimits\fi}
\makeatother

\newcommand{\cupdot}{\charfusion[\mathbin]{\cup}{\cdot}}

\newlength{\tempindent} 
\newcommand{\lazyenum}{
\setlength{\tempindent}{\parindent} 
\begin{enumerate}[leftmargin=0cm,itemindent=0.7cm,labelwidth=\itemindent,labelsep=0cm,align=left,label=(\arabic*)]
\setlength{\parskip}{\smallskipamount}
\setlength{\parindent}{\tempindent}
}
\newif
\iffurther 
\furtherfalse

\title{Probabilistic Laws on Infinite Groups}
\author{Gideon Amir}
\email{gidi.amir@gmail.com}
\address{Bar-Ilan University, Ramat Gan 52900, Israel}
\author{Guy Blachar}
\email{guy.blachar@gmail.com}
\address{Bar-Ilan University, Ramat Gan 52900, Israel}
\author{Maria Gerasimova}
\email{mari9gerasimova@mail.ru}
\address{West\"{a}lische Wilhelms-Universit\"{a}t M\"{u}nster, 48149 M\"{u}nster, Germany}
\author{Gady Kozma}
\email{gady.kozma@weizmann.ac.il}
\address{The Weizmann Institute of Science, Rehovot 76100, Israel}

\makeatletter
\@namedef{subjclassname@2020}{2020 Mathematics Subject Classification}
\makeatother
\subjclass[2020]{20P05, 20F69, 20F65, 60B15, 20E22}

\begin{document}

\maketitle
\vspace{-2.5em} 
\begin{abstract}
We study the probability that certain laws are satisfied on infinite groups, focusing on elements sampled by random walks. 
For several group laws, including the metabelian one, we construct examples of infinite groups for which the law holds with high probability, but the group does not satisfy the law virtually. On the other hand, we show that if an infinite group satisfies the law $x^2=1$ with positive probability, then it is virtually abelian.
\end{abstract}

\section{Introduction}

The commuting probability of a finite group $G$, denoted $\Pr([x,y]=1)$, is the probability that two random elements from $G$ commute with each other. The following is a well-known theorem of Gustafson \cite{Gu}:
\begin{thm*}
  If $\Pr([x,y]=1)>\frac{5}{8}$ then $G$ is abelian.
\end{thm*}
We refer to such a phenomenon as a gap: when considering the set of commuting probabilities over all finite groups as a subset of $[0,1]$, there is a gap close to $1$.

More generally, for a finite group $G$ and a word $1\neq w=w(x_1,\dots,x_d)\in\F_d$, we define the probability that $w$ is satisfied on $G$ as the probability that $d$ independent uniformly chosen elements $g_1,\dots,g_d$ of $G$ satisfy $w(g_1,\dots,g_d)=1$:
$$\Pr\left(w = 1 \textrm{ on }G\right)=\frac{\left|\set{(g_1,\dots,g_d)\in G^d\suchthat w(g_1,\dots,g_d)=1}\right|}{\left|G\right|^d}.$$
A natural question is whether for a given word $w$, these probabilities also present a similar behavior.

It is also well known (and with the same proof as Gustafson's theorem) that for a fixed $k$, the $k$-step nilpotent law $[x_1,\dots,x_k]=[[\dotsb[[x_1,x_2],x_3]\dotsb],x_k]=1$ presents a gap. 

Considering power laws $x^m=1$, several results have been established. Laffey showed the existence of a gap for the laws $x^2=1$, $x^3=1$ \cite{laf}, $x^4=1$ \cite{laf3} and $x^p=1$ for prime $p$ on non-$p$-groups \cite{laf2}. For general exponent $m$, a result of Mann and Martinez \cite{MM} proves the existence of a gap for $x^m=1$ for all groups with a fixed number of generators. Another result in this direction is the recent work of Delizia, Jezernik, Moravec and Nicotera \cite{deli} that the metabelian law $[[x,y],[z,w]]=1$ and the 2-Engel law $[[x,y],y]=1$ satisfy a gap.



Although all known cases indeed satisfy a gap for finite groups, the general case is still an open question:
\begin{ques}\label{ques:gap-fin}
  Does any law satisfy a gap for finite groups?
\end{ques}

The commuting probability of a finite group does not only present a gap, but also reveals structural information about the group. A well-known theorem of Peter Neumman \cite[Theorem 1]{Ne} states:
\begin{thm*}
  If $G$ is a finite group for which $\Pr([x,y]=1)\ge\eps>0$, then $G$ is ($\eps$-bounded)-by-abelian-by-($\eps$-bounded), i.e.\ $G$ contains subgroups $H'\vartriangleleft H<G$ such that $[G:H]$ and $\left|H'\right|$ are bounded by a function of $\eps$, and such that $H/H'$ is abelian.
\end{thm*}

We call such a behavior \emph{positivity} and say that ``the commutativity law satisfies positivity'' (see this more formally in \Dref{def:pos} below).

An interesting phenmenon is that sometimes, if a law holds with probability at least $\eps$, then \emph{another} law holds virtually. We quote two results of this type. For the law $x^2=1$, a result of Mann~\cite{Ma} shows that if $\Pr\left(x^2=1\on G\right)\ge\eps>0$, then $G$ is ($\eps$-bounded)-by-abelian-by-($\eps$-bounded). The example of dihedral groups shows that it is not true that the group is ($\eps$-bounded)-by-(satisfying $x^2=1$)-by-($\eps$-bounded).


For the $2$-step nilpotent law $[[x_1,x_2],x_3]=1$, a result of Eberhard and Shumyatsky \cite{ES} shows that such groups are not ($\eps$-bounded)-by-($2$-step-nilpotent)-by-($\eps$-bounded), but are ($\eps$-bounded)-by-($3$-step-nilpotent)-by-($\eps$-bounded).

We remark that in the nilpotent case, this phenomenon is due to our requirement that the excess parts are bounded \emph{only} by $\eps$. If one allows them to depend also on the number of generators, the picture changes completely. Indeed, Shalev \cite{Sh} showed that if $G$ is $d$-generated, and $\Pr\left([x_1,\dots,x_k]=1\on G\right)\ge\eps>0$, then $G$ is ($k$-step-nilpotent)-by-($(k,d,\eps)$-bounded).



Again, the general case remains unknown:
\begin{ques}\label{ques:pos-fin}
  If a law $w$ is satisfied on a finite group $G$ with probability larger than~$\eps$, does the group $G$ contain a large subgroup (of index bounded as a function of~$\eps$, and possibly also of the number of generators of $G$) that satisfies a law which depends only on~$w$?
\end{ques}\medskip

\subsection{Infinite groups}
One may ask for a generalization of these concepts to infinite groups. As we can no longer sample elements according to a uniform distribution, one needs to define how to measure the probability of satisfying a law on an infinite group. When the group is locally compact, for example, a natural choice for a measure is the Haar measure.

Our main interest in this paper is the case of infinite, finitely generated discrete groups. Following the ideas of \cite{AMV}, we use a sequence of probability measures $M=\set{\mu_n}$ on $G$, and define the probability that $G$ satisfies the law $w$ with respect to $M$ as
\begin{equation}\label{eq:defPM}
\Pr_M\left(w=1\on G\right)=\limsup_{n\to\infty}\mu_n\left(\set{(g_1,\dots,g_d)\suchthat w(g_1,\dots,g_d)=1}\right).
\end{equation}
If $G$ is finitely generated, natural choices for $\set{\mu_n}$ include taking a random walk measure on $G$, or a uniform measure on a sequence of balls centered at the identity and converging to $G$. We will omit $M$ if it is the random walk measure with respect to some natural generators (this is the case we are most interested in). 

In previous works by Ant\'{o}lin-Martino-Ventura \cite{AMV} and later by Tointon \cite{toin}, the commutativity probability was studied. These works prove analogous results to the finite case:
\begin{thm*}[{Tointon \cite[Theorem 1.9]{toin}}]~
  \begin{enumerate}
    \item If $\Pr\left([x,y]=1\on G\right)>\frac{5}{8}$, then $G$ is abelian. (In other words, the commutator word satisfies a gap also for infinite finitely generated groups.)
    \item If $\Pr\left([x,y]=1\on G\right)>0$, then $G$ is virtually abelian.
  \end{enumerate}
\end{thm*}
(The results of \cite{AMV} and \cite{toin} are more general than for random walks, we quoted only the random walk results for simplicity). Similar results are proved for the nilpotent case in \cite{MTVV}. \medskip



In this paper we consider other examples of laws. Our goal is to answer the two questions above for various laws. Namely, do these laws satisfy a gap for infinite finitely generated groups? If the probability that a law is satisfied on a group is positive, does the group virtually satisfy the same law, or perhaps another one?
We start with sketching an example showing that very different behaviour may occur in the infinite case.

\subsection{A lamplighter example}

Consider the group $G\coloneqq\F_2\wr\Z^5$ i.e.\ a lamplighter with the base group being $\Z^5$ and the lamps being free groups with 2 generators (exact definitions will be given in \Ssref{subsec:wreath-prods}). We will use this group to show that the metabelian law $[[x,y],[z,w]]=1$ has no gap for infinite groups. It is easy to see that $G$ does not satisfy the metabelian law virtually (thus we get that the metabelian law satisfies neither a gap nor a positivity result with the same example).

Consider first a single commutator, $[x,y]$. The walker (i.e.\ the $\Z^5$ projection) is at the identity. The lamps can be found on a random subset of places visited by the following process: do a random walk on $n$ steps. Then do another random walk on $n$ steps. Then repeat the first walker in reverse, and then the second.

The other commutator, $[z,w]$ gives a second, independent, subset of a random loop. The crucial property of $\Z^5$ that we exploit is that two such loops do not intersect (except at the identity) with positive probability (independent of the length of the loops). We conclude that with positive probability the lamps positions of $[x,y]$ and $[z,w]$ are disjoint. This, of course, means that these elements commute. This shows that the metabelian law does not satisfy positivity for infinite groups.

All this was using the natural generators of $G$. To get a no gap result we modify the generators as follows. Denote the generators of $\F_2$ by $a$ and $b$. Then the walker is allowed to act on the $a$ lamp in its current position, but the $b$ lamp is acted upon in a different position (shifted with respect to its current position by a fixed, large offset). See Figure \ref{fig:twoloops}. This modification means that even when the loops performed by the walkers intersect, with high probability you will only have lamps of one type in the intersection (at each site, either both acted with an $a$ or both acted with a $b$). This modification means that by changing the generators (their number remains~$160$) one may make $\Pr([[x,y],[z,w]]=1\on G)$ arbitrarily close to 1. This shows that the metabelian law does not even satisfy a gap. Full details of the proof can be found in \Sref{sec:commut}.

\begin{figure}
\begin{centering}
\includegraphics{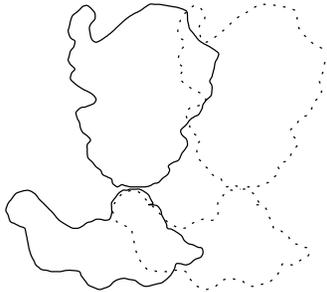}
\end{centering}
\caption{The solid lines are where $a$ was deposited, and the dotted lines are where $b$ was deposited. The proof works because the top solid loop does not intersect the bottom dotted loop, and vice versa.\label{fig:twoloops}}
\end{figure}

It is interesting to note how residual finiteness enters into the picture. Recall that the metabelian law \emph{does} satisfy a gap for finite groups \cite{deli}. This means (see \Sref{sec:res-fin}) that it satisfies a gap for residually finite groups. Thus we get a new manifestation of the fact that a lamplighter with non-commutative lamps is not residually finite (this fact is far from new, see \cite{lampresid}). 

By varying the base group one may get additional examples of laws with no gap. In \Sref{sec:self-comm-power} we take the base group to be an infinite Burnside group to get that the law $[x^k,y^k]$ (with some restrictions on $k$) does not satisfy a gap. In \Sref{sec:comm-power-balanced} we take as the base group a certain generalization of the infinite dihedral group in order to handle laws like $[x^k,[y,z]]$ (exact definition of what does ``like'' means will be found there).

\subsection{Power laws}

We present two results about power laws. The first is about the law $x^2$. We have already mentioned the result of Mann that if $\Pr(x^2=1)\ge\eps$ on a finite group $G$ then $G$ is ($\eps$-bounded)-by-abelian-by-($\eps$-bounded). This example prompted Leemann and de la Salle \cite{LdlS} to conjecture that if an infinite group $G$ satisfies $x^2=1$ with positive probability, it is virtually abelian. We prove this conjecture in \Tref{thm:x2-pos}. 

The infinite dihedral group satisfies the law $x^2=1$ with probability $\nicefrac12$ but does not satisfy this law virtually. In \Pref{non-positivity for power} we show a generalization of this fact to any power. The example is the semi-direct product $\Z^{m-1}\rtimes \Z/m\Z$ where the action of $\Z/m\Z$ is given by powers of the companion matrix of $1+\dotsb+x^m$.

Moving from positivity results to gaps, we present one result on this topic. We show that the law $x^3=1$ satisfies a gap when $\limsup$ is replaced by $\liminf$ in (\ref{eq:defPM}) (\Tref{thm:x3-gap}). The case of powers bigger than 3 remains open.

\subsection{Additional results}

For the case of \emph{residually finite} groups, we show in \Sref{sec:res-fin} that (under certain conditions), the answer to both positivity and gap questions is the same as the finite case. 

In \Sref{sec:derived} we show that for laws of the form $w'=[w(x_1,\dots,x_n),x_{n+1}]$, a positive answer for $w$ to one of these questions implies a positive answer for $w'$ to the same question, generalizing the case of the $k$-step nilpotent law.

We focused so far on probabilities with respect to random walk, but it seems quite natural to also consider probabilities taken uniformly on balls in the Cayley graph of $G$. We believe that these measures are less natural, and we demonstrate that with one result. In \Sref{sec:uniform-measures-balls} we show that many laws cannot satisfy a gap result with respect to uniform elements of balls. The example is simply the product with a finite group which does not satisfy the law.

\subsection{A lemma on random walks}
We end this introduction mentioning one lemma for readers interested in random walks on groups. A well-known open problem in this area is whether random walk on any transient group spends on more than $r^2$ time in a ball of radius $r$, with the best result known to us being that of Lyons, Peres, Sun and Zheng \cite{LPSZ20}. We use the techniques of \cite{LPSZ20} for the \emph{path of a word}, namely for the elements of a group visited by $w(R^{(1)},\dotsc,R^{(d)})$, where $w$ is some word in $d$ letters and $R^{(1)},\dotsc,R^{(d)}$ are $k$ independent random walks. See the exact statement in \Sref{sec:occupation-measure}.

\subsection{Acknowledgements}
We would like to thank Mikael de la Salle for an important suggestion for the proof of \Tref{thm:x2-pos}. We thank Yehuda Shalom for some useful references, and Denis Osin and Andreas Thom for their comments and remarks.

During this research G.A.\ and G.B.\ were supported by Israeli Science Foundation grant \#957/20. G.B.\ was also supported by the Bar-Ilan President's Doctoral Fellowships of Excellence. G.K.\ was supported by the Israel Science Foundation and by the Jesselson Foundation. M.G.\ was supported by the DFG -- Project-ID 427320536 -- SFB 1442, and under Germany's Excellence Strategy EXC 2044 390685587, Mathematics Münster: Dynamics--Geometry--Structure.


\section{Preliminaries}\label{sec:prelim}

\subsection{The probability of satisfying a law}

By a \textbf{word} we mean a nontrivial element of the free group $\F_d$. Let $x_1,\dots,x_d$ denote the standard generators of~$\F_d$. Any word $w\in\F_d$ defines the \textbf{word map} on any group $G$, which we also denote $w\colon G^d\to G$, defined as follows: for $g_1,\dots,g_d$, the value of $w(g_1,\dots,g_d)$ is given by substituting $g_1,\dots,g_d$ for $x_1,\dots,x_d$. We interpret any word as a law on groups by considering the equation $w(x_1, \dots x_d)=1$ (where $1$ is the identity element of~$G$).

For a measure $\mu$ on a group $G$, write $\mu$ also for the measure $\mu\times\cdots\times\mu$ induced on $G^d$. We use the ideas presented in \cite{MTVV}:

\begin{defn}
Let $G$ be a group, and let $w\in\F_d$ be a word.
\begin{enumerate}
  \item Let $\mu$ be a probability measure on $G$. We define the \textbf{probability that $G$ satisfies $w$ with respect to $\mu$} as
      $$\Pr_{\mu}\left(w=1\on G\right)=\mu\left(\set{(g_1,\dots,g_d)\in G^d\suchthat w(g_1,\dots,g_d)=1}\right).$$
  \item Let $M=\set{\mu_n}$ be a sequence of probability measures on $G$. We define the \textbf{probability that $G$ satisfies $w$ with respect to $M$} as
      $$\Pr_M\left(w=1\on G\right)=\limsup_{n\to\infty}\,\Pr_{\mu_n}\left(w=1\on G\right).$$
\end{enumerate}
\end{defn}

Let $G$ be a finitely generated group. We present here several choices for sequences of probability measures on $G$, which we will consider throughout the paper.

Our main interest will be \textbf{random walk measures}. These are sequences of the form $M=\set{\mu_n=\mu^{*n}}$, where $\mu$ is a finitely supported generating probability measure on a group $G$ with $\mu(1)>0$. In other words, $\mu_n$ is the probability measure induced by a random walk with step distribution $\mu$ after $n$ steps. It will occasionally be important to consider random walks also as a stochastic process (i.e.\ a measure on paths in the group) and then we will always use a right random walk, i.e.\ $R_{n+1}=R_nx_n$ where the $x$ are i.i.d.\ with law $\mu$.

Another type of sequences of interest is given in \cite{toin} as follows:
\begin{defn}
  We say that a sequence of probability measures $M=\set{\mu_n}$ on $G$ \textbf{measures index uniformly}, if $\mu_n(xH)\to 1/[G:H]$ uniformly over all $x\in G$ and all subgroups $H\leq G$. Similarly, for a sequence of random variables taking value in $G$ we say that they measure index uniformly if their laws measure index uniformly.
\end{defn}
For example, (lazy) random walk measures satisfy this property \cite[Theorems 1.11 and 1.12]{toin}. 

An additional way to define a sequence of measures on a group is by taking the uniform measures on balls centered at the identity. Formally, we take a finite symmetric generating set $S$ of $G$ which contains the identity, and let $\mu_n$ be the uniform measure on $S^n$ (i.e.\ the ball of radius $n$ around the identity in the Cayley graph). While such sequences of measures usually do not measure index uniformly (for instance, see \cite[Remark after Theorem 1.12]{toin}), they are of interest on their own. We discuss these measures in \Sref{sec:uniform-measures-balls}.

\subsection{Gaps and positivity}

In this paper, we deal with two properties of laws. The first one concerns the values of the probabilities that the law is satisfied:

\begin{defn}
Fix a family of sequences of measures $\mathcal{M}$, and let $w\in\F_d$ be a word.
\begin{enumerate}
  \item We say that $w$ satisfies a \textbf{gap} with respect to $\mathcal{M}$, if there is a function $\phi\colon\N\to(0,1)$ such that for any $r$-generated group $G$ and sequences of measures $M\in\mathcal{M}$ on $G$, if $\Pr_M\left(w=1\on G\right)>1-\phi(r)$, then $G$ satisfies~$w$. In this case we also say that $w$ satisfies a \textbf{$\phi$-gap}.
  \item If $w$ satisfies a gap with respect to a constant function $\phi(r)=\eps$, we say that $w$ satisfies a \textbf{strong gap} (or an \textbf{$\eps$-gap}) with respect to $\mathcal{M}$. In this case, the probability that $w$ is satisfied cannot lie in the interval $(1-\eps,1)$.
\end{enumerate}
\end{defn}

We note that this definition might depend on the family of probability measures one considers. Unless explicitly mentioned, we will assume that $\mathcal{M}$ is the family of random walk measures.

Summarizing the results mentioned in the introduction in this language, for finite groups it is known the the laws $[x,y]=1$, $x^2=1$, $x^3=1$, $x^4=1$, $[x,y,y]=1$ and $[[x,y],[z,w]]=1$ satisfy a strong gap, while $x^n=1$ satisfies a weak gap. For finitely generated groups, it is known that $[x,y]=1$ and $[x_1,\dots,x_k]=1$ satisfy a strong gap.\medskip

We move to discuss the second property of laws we will study. We say that a group \textbf{virtually} satisfies a law $w$, if its has a finite index subgroup which satisfies~$w$. If a group virtually satisfies a law, it has a positive probability to satisfy it. This leads to the following:
\begin{defn}\label{def:pos}
Fix a family of sequences of measures $\mathcal{M}$, and let $w\in\F_d$ be a word.
\begin{enumerate}
  \item We say that $w$ satisfies a \textbf{positivity property} with respect to $\mathcal{M}$, if there is a function $\psi\colon(0,1)\times\N\to\N$ such that the following holds: for any $\eps\in(0,1)$, any $r$-generated group $G$ and any $M\in\mathcal{M}$ on $G$, if $\Pr_M\left(w=1\on G\right)\ge\eps$, then $G$ has a subgroup of index $\leq\psi(\eps,r)$ which satisfies~$w$.
  \item We say that $w$ satisfies a \textbf{strong positivity property} with respect to $\mathcal{M}$, if it has a positivity property with respect to a function $\psi(\eps,r)$ which depends only on $\eps$.
\end{enumerate}
\end{defn}

For example, Tointon \cite[Theorem 1.9]{toin} showed that the commutativity law satisfies a strong positivity property.

\begin{rem*}
  If $G$ virtually satisfies a law, it must completely satisfy some other law. In particular, if $G$ contains a non-abelian free subgroup, it cannot virtually satisfy any law.

  Indeed, let $\Lambda\leq G$ be a finite index subgroup that satisfies a law $w$. There exists a normal subgroup $N\vartriangleleft G$ with $N\leq\Lambda$ and $[G:N]=k<\infty$. Therefore $g^k\in N\leq\Lambda$ for all $g\in G$, so $G$ satisfies the law $w(x_1^k,\dots,x_d^k)=1$.
\end{rem*}

\begin{rem}\label{subgroup}
  If $G$ virtually satisfies a law, then each subgroup $H$ of $G$ also virtually satisfies this law. Indeed, let $\Lambda$ be a finite index subgroup of $G$ that satisfies this law, then $\Lambda \cap H$ is a finite index subgroup in $H$ that also satisfies this law.
\end{rem}

\subsection{Wreath products}\label{subsec:wreath-prods}

A primary construction that we will use in this paper is the wreath product of two groups. We recall here the definition and a basic property regarding laws of wreath products.

Let $G,H$ be two groups. The \textbf{wreath product} (or \textbf{lamplighter group}) $H\wr G$ is the semidirect product $\left(\bigoplus_{g\in G} H\right)\rtimes G$, where $G$ acts on $\bigoplus_{g\in G} H$ by permuting the indices, i.e.
$$g\cdot (h_a)_{a\in G}=(h_{g^{-1}a})_{a\in G}.$$
The group $G$ is called the \textbf{base group} of $H\wr G$, and $H$ is called the \textbf{lamp group} of $H\wr G$. Each element of $H\wr G$ can be written in the form $(L,g)$, where $g\in G$ is the \textbf{lamplighter position}, and $L\in\bigoplus_{g\in G} H$ is the \textbf{lamp configuration}. For an element $(L,g)$ we define its \textbf{support} $\supp (L,g)=\{a\in G \suchthat L(a)\neq 1\}$.

We make the following observation:

\begin{prop}
  If $H$ does not satisfy a law $w\in\F_d$, then for any infinite group~$G$ the lamplighter group $H\wr G$ does not satisfy the law $w$ virtually.
\end{prop}
\begin{proof}
  Let us take $\bigoplus_{g\in G} H$, which we identify with the subgroup of all elements with trivial lamplighter position. Assume by contradiction that there exists a finite index subgroup $\Lambda<H\wr G$ that satisfies the law $w=w(x_1, \dots, x_d)$. Then note that $\Lambda \cap \bigoplus_{g\in G} H$ is a finite index subgroup in $\bigoplus_{g\in G} H$. Let us take representatives $t_1, \dots, t_n$ of the left cosets of $\bigoplus_{g\in G} H /(\Lambda \cap \bigoplus_{g\in G}H).$ We take the set $A= \cup_{i=1}^n \supp(t_i) \subset G$ and consider $g_0 \in G\setminus A$. Let us consider $h_1, \dots h_d \in H$ such that $w(h_1, \dots, h_d) \neq 1$, and take $y_1, \dots, y_d \in \bigoplus_{g\in G} H$ such that $y_i(g_0)=h_i$ for $i=1, \dots, d$. Then for any $i=1, \dots, d$ there exists $j\in \{1, \dots, n\}$ such that $z_i:=t_j^{-1}y_i \in \Lambda \cap \bigoplus_{g\in G} H$, so $z_i(g_0)=h_i$. Therefore, since $w(z_1, \dots, z_d)(g_0)=w(h_1, \dots, h_d)$, we get that $w(z_1, \dots, z_d)\neq 1$. A contradiction.
\end{proof}

\subsection{Notation} A ``word'' is a sequence $w$ of letters in some alphabet. We denote by $|w|$ the number of letters in $w$. Let $G$ and $H$ be groups and let $g\in G$ and $h\in H$. Then the notation $h\delta_g$ refers to the element of $\bigoplus_{g\in G}H$ which takes value $h$ at $g$ and takes the unit of $H$ at any $g'\ne g$.

We use the following standard group notation. We use $C_G(x)$ for the centralizer of $x$ in $G$, namely, $\{y\in G:[x,y]=1\}$. We use $C_G(H)=\bigcap_{h\in H}C_G(h)$. We use $Z(G)$ for the center of $G$, namely, $\{g\in G:C_G(g)=G\}$.

We use $c$ and $C$ to denote constants that depend only on the law in question, and, for infinite groups, on the walk. Their value may change from line to line, and even within the same line. We use $c$ for constants which are ``sufficiently small'' and $C$ for constants which are ``sufficiently large''. When we want to specify specific constants we will  use $c'$, $c''$, etc\@. These change value only from lemma to lemma.

\section{Residually finite groups}\label{sec:res-fin}


In this section we study the questions of gap and positivity for laws on residually finite groups. We show that, when working with sequences of probability measures that measure index uniformly, these questions are equivalent to the same questions on finite groups.

Let $G$ be a group, and let $N\vartriangleleft G$ be a normal subgroup. Any probability measure~$\mu$ on $G$ induces a probability measure $\overline{\mu}$ on $G/N$ by taking the pushforward of $\mu$ with respect to the natural projection $G\to G/N$. The following lemmas follow directly from the definition of measuring index uniformly and the pushforward measure:

\begin{lem}
Let $G$ be a finite group, and let $M=\set{\mu_n}$ be a sequence of probability measures on $G$ that measures index uniformly. Let $U_G$ denote the uniform probability measure on $G$. Then $\mu_n\to U_G$.
\end{lem}

\begin{lem}\label{lem:prob-quo}
Let $G$ be a group, and let $N\vartriangleleft G$ be a normal subgroup. Let $M=\set{\mu_n}$ be a sequence of probability measures on $G$, and let $\overline{M}=\set{\overline{\mu_n}}$. Let $w$ be a word. Then
$$\Pr_M\left(w=1\;\mathrm{on}\;G\right)\leq\Pr_{\overline{M}}\left(w=1\;\mathrm{on}\;G/N\right).$$
\end{lem}
\begin{proof}
For $g_1,\dots,g_d\in G$, if $w(g_1,\dots,g_d)=1$ in $G$, then $w(g_1N,\dots,g_dN)=1$ in~$G/N$, so the probability of satisfying $w$ cannot decrease.
\end{proof}

We can now combine the above lemmas to study the general case of residually finite groups. Recall that $N\vartriangleleft_f G$ if $N$ is a normal subgroup of $G$ of finite index.

\begin{prop}\label{prop:res-fin-inf}
    Let $G$ be a finitely generated residually finite group, and let $M$ be a sequence of probability measures on $G$ that measures index uniformly. Then
    $$\Pr_M\left(w=1\;\mathrm{on}\;G\right)=\inf\set{\Pr\left(w=1\;\mathrm{on}\;G/N\right)\suchthat N\vartriangleleft_f G}.$$
\end{prop}

\begin{proof}
    The left hand side is not bigger than the right hand side by \Lref{lem:prob-quo}. For the opposite direction, as $G$ is finitely generated and residually finite, one can find a sequence of finite indexed normal subgroups $N_k\vartriangleleft_f G$ such that $N_{k+1}\sub N_k$ and $\bigcap_{k=1}^{\infty}N_k=1$. Therefore
    $$\set{w=1\on G}=\bigcap_{k=1}^{\infty}\set{w=1\on G/N_k},$$
    so by continuity of probability we have
    $$\Pr_M\left(w=1\on G\right)=\lim_{k\to\infty}\Pr\left(w=1\on G/N_k\right)\ge\inf\set{\Pr\left(w=1\on G/N\right)\suchthat N\vartriangleleft_f G}.$$
\end{proof}

\begin{thm}
  Let $w\in\F_d$ be a word, and let $\mathcal{M}$ denote a family of sequences of probability measures that measure index uniformly.
  \begin{enumerate}
    \item If $w$ satisfies a gap for all finite groups, then $w$ satisfies a gap for all residually finite groups with respect to $\mathcal{M}$, with the same gap function $\eps(r)$.
        \item If $w$ satisfies a positivity result for all finite groups, then $w$ satisfies a positivity result for all residually finite groups with respect to $\mathcal{M}$.
  \end{enumerate}
\end{thm}

It follows that if $w$ satisfies a strong gap for all finite groups then it satisfies a strong gap for residually finite groups, because strong gap is simply a weak gap with a gap function $\eps(r)$ bounded below.
\begin{proof}
  Let $G$ be an $r$-generated residually finite group, and let $M=\set{\mu_n}$ be a sequence of probability measures on $G$ that measures index uniformly.

  \lazyenum
    \item This is a direct corollary of \Pref{prop:res-fin-inf}.


    \item Assume that $w$ satisfies a weak positivity property for all finite groups. That is, there is a function $\psi\colon(0,1)\times\N\to\N$ such that for any $\eps\in(0,1)$ and any finite $r$-generated group $\Lambda$, if $\Pr(w=1\on\Lambda)\ge\eps$ then $\Lambda$ has a subgroup of index $\leq\psi(\eps,r)$ satisfying $w$.

        If $\Pr_M\left(w=1\on G\right)\ge\eps$, the same holds for $G/N$ for any finite indexed normal subgroup $N\vartriangleleft G$. Therefore, for any such $N$, the group $G/N$ contains a subgroup of index $\leq\psi(\eps,r)$ which satisfies $w$, so there is a subgroup $N\leq H\leq G$ such that $[G:H]\leq\psi(\eps,r)$ and $w(H,\dots,H)\sub N$.

        Let $\mathcal{T}$ denote the set of subgroups of $G$ of index $\leq\psi(\eps,r)$. As $G$ is finitely generated, it has finitely many such subgroups and their number only depends on $r$ and $\psi(\eps,r)$ (see, e.g., \cite[Corollary 1.1.2]{lubotzbook}). Hence $K=\bigcap_{H\in\mathcal{T}}H$ is a finite index subgroup of $G$, and its index is bounded by a function of $\eps$ and $r$. Note that $w(K,\dots,K)\sub\bigcap_{N\vartriangleleft_f G}N=\set{e}$, so~$K$ satisfies $w$, as required.\qedhere
  \end{enumerate}
\end{proof}

\section{Derived laws}\label{sec:derived}

Fix a word $w=w(x_1,\dots,x_d)\in\F_d$. We define the \textbf{derived word of $w$} to be $w'=[w(x_1,\dots,x_d),x_{d+1}]\in\F_{d+1}$. In this section we show that if $w$ satisfies a gap or positivity result, then so does $w'$. This is a result with similar flavor to results in \cite{Sh}. We now state our main theorem:

\begin{thm}\label{thm:derived}
Let $w\in\F_d$ be a word, and consider $w'\in\F_{d+1}$. Let $\mathcal{M}$ be the family of sequences of probability measures that measure index uniformly. Then:
\begin{enumerate}
    \item\label{item:derived-gap} If $w$ satisfies a gap with respect to $\mathcal{M}$, then $w'$ satisfies a gap with respect to $\mathcal{M}$;
    \item\label{item:derived-pos} If $w$ satisfies a positivity result with respect to $\mathcal{M}$, then $w'$ satisfies a positivity result with respect to $\mathcal{M}$.
\end{enumerate}
\end{thm}

Our strategy is similar to the proofs of \cite[Proposition 2.3]{kocs} and \cite[Proposition 1.8]{Sh}. We first prove a version of \cite[Proposition 2.1]{toin}:

\begin{lem}\label{lem:markov}
Let $G$ be a group, and let $x_n$ and $y_n$ be two sequences of random variables taking value in $G$. Assume that $y_n$ measures index uniformly and that
$$\limsup_{n\to\infty}\Pr([x_n,y_n]=1)>\eps.$$
Then, for any $\delta>0$,
$$\limsup_{n\to\infty}\Pr([G:C_G(x_n)]<1/\delta)>\frac{\eps-\delta}{1-\delta}.$$
\end{lem}
As usual, $C_G(x)$ is the centralizer of $x$ in $G$.
\begin{proof}
By definition, there is a sequence $\set{n_k}$ of nonnegative integers such that $\Pr([x_{n_k},y_{n_k}]=1)>\eps$, i.e.\ $\Pr[y_{n_k}\in C_G(x_{n_k})]>\eps$. Write
$$A=\set{x\in G\suchthat [G:C_G(x)]<1/\delta}.$$
Then
\begin{align*}
    \eps&<\Pr(x_{n_k} \in A) +\E[y_{n_k}\in C_G(x_{n_k})\,|\, x\notin A]\cdot\Pr(x_{n_k}\notin A)\leq\\
    &\leq\Pr(x_{n_k} \in A)+
    \left(\frac{1}{[G:C_G(x)]}+o(1)\right)(1-\Pr(x_{n_k}\in A))\leq\\
    &\leq\Pr(x_{n_k}\in A)+\delta(1-\Pr(x_{n_k}\in A))+o(1)
\end{align*}
where the second inequality used the fact that $y_n$ measures index uniformly. The conclusion follows.
\end{proof}

\smallskip 

\begin{proof}[Proof of \Tref{thm:derived}]~
\lazyenum\item
Suppose that $w$ satisfies a gap with respect to a function $\phi$. We may assume that $\phi(r)>\frac{1}{2}$ for all $r$. Let $G$ be an $r$-generated group, and let $x_1 ,\dotsc,x_{d+1}$ be sequences of random variables measuring index uniformly on $G$ such that 
$$\limsup_{n\to\infty}\Pr\left([w(x_1,\dotsc,x_d),x_{d+1}]=1\on G\right)=\eps>\frac{\phi(r)+1}{2}$$ 
(each $x_i$ is a sequence of variables depending on $n$, but we suppress the dependence on $n$ in the notation). We apply \Lref{lem:markov} with $x_{\textrm{\Lref{lem:markov}}}=w(x_1,\dotsc,x_d)$, $y_{\textrm{\Lref{lem:markov}}}=x_{d+1}$ and $\delta=\frac{1}{2}$. Note that it is important here that in \Lref{lem:markov} only $y$ needs to measure index uniformly, and $x$ does not need to do that. We get
 $$       
          \limsup_{n\to\infty}\Pr\left([G:C_G(w(x_1,\dots,x_d))]<2\right)>\frac{\eps-\frac{1}{2}}{1-\frac{1}{2}}>\phi(r).
$$ 
Of course, if $[G:C_G(w(x_1,\dotsc,x_d)]<2$ then in fact $w(x_1,\dotsc,x_d)\in Z(G)$, the center of $G$.

        Consider the $r$-generated  group $\overline{G}=G/Z(G)$. Using the induced sequences of random variables $\overline{x_i}$, we have $\Pr\left(w(\overline{x_1},\dotsc,\overline{x_d})=1\on\overline{G}\right)>\phi(r)$; therefore $\overline{G}$ satisfies $w$. But then $w(x_1,\dots,x_n)\in Z(G)$ for all $x_1,\dots,x_n\in G$, so $G$ satisfies $w'$.

        From the proof we also see that if $w$ satisfies a strong gap, i.e.\ $\phi$ is constant, then so does $w'$.
\item
Assume that $w$ satisfies a positivity property with respect to a function $\psi\colon(0,1)\times\mathbb{N}\to\mathbb{N}$. Let $\eps>0$, let $G$ be an $r$-generated group and let $x_1,\dotsc,x_{d+1}$ be random variables on $G$ measuring index uniformly such that 
$$\Pr\left(w'(x_1,\dotsc,x_{d+1})=1\on G\right)>\eps.$$ 
Using \Lref{lem:markov} as in the previous part we get
$$\limsup_{n\to\infty}\Pr\left([G:C_G(w(x_1,\dots,x_d))]<\frac{2}{\eps}\right)>\frac{\eps-\frac{\eps}{2}}{1-\frac{\eps}{2}}=\frac{\eps}{2-\eps}.$$
        Let $N$ denote the intersection of all subgroups of $G$ of index at most $\frac{2}{\eps}$. This is a normal subgroup of $G$, and its index is bounded as a function of~$r$ and $\frac{2}{\eps}$. Then we have
        $$\limsup_{n\to\infty}\Pr\left(w(x_1,\dots,x_d)\in C_G(N)\right)>\frac{\eps}{2-\eps}.$$
        Consider the group $\overline{G}=G/C_G(N)$. By the above,
        $$\limsup_{n\to\infty}\Pr\left(w(x_1,\dots,x_d)C_G(N)=C_G(N)\right)>\frac{\eps}{2-\eps}.$$
        Therefore $G/C_G(N)$ contains a subgroup $H/C_G(N)$ of index bounded by $\psi(\frac{\eps}{2-\eps},r)$ that satisfies $w$. We therefore conclude that $H\cap N$ satisfies $w'$, and its index is bounded as a function of $\eps$ and $r$.\qedhere
\end{enumerate}
\end{proof}

\section{Power laws}\label{sec:power-laws}

\subsection{Positivity for \texorpdfstring{$x^2=1$}{x squared}}



In this section we prove 
\begin{thm}\label{thm:x2-pos}
Let $G$ be a finitely generated group. If
$$\Pr\left(x^2=1\on G\right)>0$$
with respect to some random walk, then $G$ is virtually abelian.
\end{thm}

Recall that our standing assumption is that all random walks are lazy, i.e.\ have positive probability to stay in the same place for one turn. This fact is not crucial (the theorem holds with essentially the same proof without this assumption), but is convenient because it implies that random walk measures index uniformly.

For the rest of this subsection, we prove the above theorem. Let $G$ be a finitely generated group, and let $R_n$ be a (lazy) random walk on $G$ such that
$$\limsup_{n\to\infty}\Pr(R_n^2=1)=\alpha>0.$$
Take a sequence $\set{m_k}$ such that $\Pr(R_{m_k}^2=1)\ge\frac{\alpha}{2}$ for all $k$.

\begin{lem}\label{lem:RnRn+s}
There exists $k_0$ such that for any $k\ge k_0$,
$$\limsup_{s\to\infty}\Pr(R_{m_k}^2=1,R_{m_k+s}^2=1)>\frac{\alpha^2}{16}.$$
\end{lem}
\begin{proof}
Suppose on the contrary that no such $k_0$ exists, i.e.\ there are infinitely many values of $k$ such that
\begin{equation}\label{eq:x2-lem1}
\limsup_{s\to\infty}\Pr(R_{m_k}^2=1,R_{m_k+s}^2=1)\leq\frac{\alpha^2}{16}.
\end{equation}
Write $M=\ceil{\frac{4}{\alpha}}$. We construct a sequence~$\set{k'_i}_{i=1}^M$ as follows. Let $k'_1$ be the first index satisfying~\eqref{eq:x2-lem1}. Let $k'_2$ be the first index satisfying~\eqref{eq:x2-lem1} and such that
$$\Pr(R_{m_{k'_1}}^2=1,R_{m_{k'_2}}^2=1)<\frac{\alpha}{4}$$
(which exists because $\frac{\alpha}{4}>\frac{\alpha^2}{16}$). Let $k'_3$ be the first index satisfying~\eqref{eq:x2-lem1} and such that
$$\Pr(R_{m_{k'_j}}^2=1,R_{m_{k'_3}}^2=1)<\frac{\alpha}{8}\textrm{ for }j=1,2.$$
We continue inductively to define $\set{k'_i}$ for $1\leq i\leq M$, where $k'_i$ is chosen as the first index satisfying~\eqref{eq:x2-lem1} such that
$$\Pr(R_{m_{k'_j}}^2=1,R_{m_{k'_i}}^2=1)<\frac{\alpha}{4(i-1)}\textrm{ for }j=1,\dots,i-1.$$
Such indices exist since for any $i\leq M$ one has $\frac{\alpha}{4(i-1)}>\frac{\alpha^2}{16}$.

As $\Pr(R_{m_{k'_i}}^2=1)\ge\frac{\alpha}{2}$ for any $1\leq i\leq M$, it follows that
$$\Pr(R_{m_{k'_i}}^2=1,R_{m_{k'_1}}^2\neq 1,\dots,R_{m_{k'_{i-1}}}^2\neq 1)>\frac{\alpha}{2}-(i-1)\cdot\frac{\alpha}{4(i-1)}=\frac{\alpha}{4}.$$
But this is a sequence of $M=\ceil{\frac{4}{\alpha}}$ pairwise disjoint events with probability larger than $\frac{\alpha}{4}$, a contradiction.
\end{proof}

Write $\beta=\frac{\alpha^2}{16}$.
\begin{lem}
$$\limsup_{n\to\infty}\Pr\left([G:C_G(R_n)]\leq\frac{2}{\beta^2}\right)\ge\frac{\beta^2}{2}.$$
\end{lem}
As usual, $C_G(R_n)$ is the centralizer of $R_n$ in $G$.
\begin{proof}
Fix one $n$ such that
\begin{equation}\label{eq:RnRn+s}
\limsup_{s\to\infty}\Pr(R_n^2=1,R_{n+s}^2=1)>\beta>0.
\end{equation}
By the previous lemma, there are infinitely many such $n$. Then there is a sequence $s_1<s_2<\cdots$ such that
$$\Pr(R_n^2=1,R_{n+s_i}^2=1)\ge\beta.$$
Write $u=R_n^{-1}R_{n+s_i}$, which has the same distribution as $R_{s_i}$ and is independent of $R_n$. Note that the equations $R_n^2=R_{n+s_i}^2=1$ imply that $u^{-1}=R_nuR_n^{-1}$, so
$$\Pr(u^{-1}=R_nuR_n^{-1})\ge\beta.$$
If $R'_n$ is an independent copy distributed like $R_n$, we have
$$\Pr(u^{-1}=R_nuR_n^{-1}=R'_nu(R'_n)^{-1})\ge\beta^2,$$
so
$$\Pr([(R'_n)^{-1}R_n,u]=1)\ge\beta^2.$$
As $(R'_n)^{-1}R_n$ has the same distribution as $R_{2n}$, and both are independent of $u$,
$$\Pr([R_{2n},u]=1)\ge\beta^2.$$
Hence it follows (similarly to \Lref{lem:markov}),
$$\Pr\left(\Pr\left([R_{2n},u]=1\suchthat R_{2n}\right)\ge\frac{1}{2}\beta^2\right)\ge\frac{1}{2}\beta^2.$$
Since $u$ is distributed like a random walk of length $s_i$ and is independent of $R_{2n}$, taking $i$ to infinity gives
$$\Pr\left([G:C_G(R_{2n})]\leq\frac{2}{\beta^2}\right)\ge\frac{1}{2}\beta^2.$$
By \Lref{lem:RnRn+s} there are infinitely many $n$ such that (\ref{eq:RnRn+s}) holds. This proves the lemma.
\end{proof}

We are now ready to prove the theorem.

\begin{proof}[Proof of \Tref{thm:x2-pos}]
Since $G$ is finitely generated, it has finitely many subgroups of any finite index. Therefore, writing
$$Q=\bigcap_{H\leq G,[G:H]\leq 2/\beta^2}H,$$
we have $[G:Q]<\infty$.

Now,
$$\Pr(R_{2n}\in C_G(Q))=\Pr([R_{2n},Q]=1)\ge\frac{1}{2}\beta^2,$$
so $[G:C_G(Q)]\leq\frac{2}{\beta^2}<\infty$.
Therefore $Q\cap C_G(Q)$ is an abelian subgroup of $G$ of finite index, as required.
\end{proof}

We give an immediate corollary from the above theorem.

\begin{cor}
\begin{enumerate}
    \item The law $x^2=1$ satisfies a gap.
    \item The value of $\limsup_{n\to\infty}\Pr(R_n^2=1)$ does not depend on the random walk's step distribution $\mu$, and it is a limit rather than a $\limsup$.
\end{enumerate}
\end{cor}

\begin{proof}
We first prove that $x^2=1$ satisfies a gap. By \Tref{thm:x2-pos}, any group which satisfies $x^2=1$ with positive probability is virtually abelian, and thus residually finite. Therefore the gap follows from the gap in the finite case.

For the second part, if $\limsup_{n\to\infty}\Pr(R_n^2=1)>0$ for some random walk $R_n$ on $G$, \Tref{thm:x2-pos} shows that $G$ is virtually abelian. In particular, $G$ is virtually nilpotent, so the corollary follows from \cite[Corollary 1.14]{MTVV}.
\end{proof}

The proof of \Tref{thm:derived}(\ref{item:derived-pos}) can be generalized to show the following:
\begin{prop}
    If a group $G$ satisfies $[x^2,y]$ probabilistically, $G$ is virtually $2$-step nilpotent.
\end{prop}
The changes required are straightforward and we omit the details.

\subsection{Strong gap for \texorpdfstring{$x^3=1$}{x to the power 3}}

In this subsection, our aim is to prove the following result:

\begin{thm}\label{thm:x3-gap}
Consider a finitely generated group $G$ with a finite symmetric generating set $S$ containing the identity, and let $R_n$ be a random walk on $G$. Then there exists $\eps>0$ such that if
$$\liminf_{n\to\infty}\,\Pr\left(R_n^3=1\right)>1-\eps$$
then $x^3=1$ holds in $G$ (and therefore $G$ is finite).
\end{thm}

Our proof also finds an explicit value of $\eps$ such that the above theorem holds. However, we believe that this $\eps$ is not optimal, and that the theorem can be proved for $\limsup$ instead of $\liminf$.\medskip

We prove this result in several steps. 
Let $G$ be a group, and let $x,y,z,w\in G$. A product of $x^{\pm 1},y^{\pm 1},z^{\pm 1},w^{\pm 1}$ is called \textbf{simple} if any letter appears at most once (thus there are only finitely many simple products of any four elements). 
We say that $x,y,z,w$ are \textbf{good} if any simple product of them has order $3$. For example, if $x,y,z,w$ are good then $(z^{-1}xw)^3=1$ and $(yx^{-1}w^{-1}z)^3=1$, but not necessarily $(xyx^{-1}z)^3=1$.

Note that the assumption of the theorem suggest that if $x,y,z,w$ are independent random walks with respect to $S$, then $x,y,z,w$ are good with high probability. Indeed, any simple product of $x,y,z,w$ is also distributed like a random walk with respect to $S$ (with a different number of steps), and therefore has order $3$ with high probability.

Our method of proof is as follows. We prove that if $x,y,z,w$ are good, then $[x,y,z,w]=1$ (we originally verified this with a GAP program, but here we give a human proof of this fact). Therefore, if $\Pr\left(R_n^3=1\right)>1-\eps$ for small enough $\eps$, the group $G$ is nilpotent with positive probability. Therefore, $G$ is virtually nilpotent, so it is residually finite, and as $x^3=1$ has a gap on finite groups we are finished.\medskip

Recall the following standard notation. For any $a,b\in G$ write $a^b=bab^{-1}$, and for any $a_1,\dots,a_n\in G$ we defined inductively $[a_1,\dots,a_n]=[[a_1,\dots,a_{n-1}],a_n]$.

\begin{lem}\label{lem:[ab,a]}
  If $b^3=(ab)^3=(ab^{-1})^3=1$, then $[a^b,a]=1$.
\end{lem}
\begin{proof}
  This follows from the identity
  $$[a^b,a]=\left(\left(ab^{-1}\right)^b\right)^3\left(b^2a^{-1}\right)^3\left(\left(b^{-1}\right)^a\right)^3$$
  which holds in any group.
\end{proof}

\begin{lem}\label{cor:x3-2eng-invcom}
  If $[a^b,a]=1$ and $a^3=1$ then $[a,b]^3=1$, $[b,a,a]=1$ and $[a,b]^{-1}=[a^{-1},b]$.
\end{lem}
\begin{proof}
  Note that $[b,a]=a^ba^{-1}$, so
  \begin{align}
    [a,b]^3&=\left(a^ba^{-1}\right)^3=(a^b)^3a^{-3}=1\label{eq:[a,b]^3}\\
    [b,a,a]&=[[b,a],a]=[a^ba^{-1},a]=1\nonumber\\
    [a,b]^{-1}&=bab^{-1}a^{-1}=a^ba^{-1}=a^{-1}a^b=a^{-1}bab^{-1}=[a^{-1},b].\label{eq:[a,b]^-1}
  \end{align}
  The lemma is thus proved.
\end{proof}

\begin{lem}\label{lem:x3-comm-trans}
  If $[a^b,a]=[a^c,a]=[a^b,a^c]=[b^a,b]=[(bc)^a,bc]=[c^b,c]=[c^{ba},c]=1$, then $[a,b,c]^{-1}=[a,c,b]$.
\end{lem}
\begin{proof}
  In any group 
  \begin{equation}\label{eq:[a,bc]}
  [a,bc]=[a,b][a,c]^b \textrm{ and } [ab,c]=[b,c]^a[a,c],
  \end{equation}
  so
  $$[a,bc,c]=[[a,b][a,c]^b,c]=[[a,c]^b,c]^{[a,b]}[a,b,c].$$
  Note that $[a,c]^b=(c^ac^{-1})^b=c^{ba}(c^{-1})^b$, so by the assumptions $[[a,c]^b,c]=1$, and thus
  $$[a,bc,c]=[a,b,c].$$
  In a similar manner,
  $$[a,bc,b]=[[a,b][a,c]^b,b]=[[a,c]^b,b]^{[a,b]}[a,b,b]\stackrel{(*)}{=}[[a,c]^b,b]^{[a,b]}=[[a,c]^{[a,b]b},b^{[a,b]}].$$
  As $[b^a,b]=1$, the elements $b$ and $[a,b]$ commute, which gives $(*)$ and also $b^{[a,b]}=b$. In addition, $[a,b]=a(a^{-1})^b$ and $[a,c]=a(a^{-1})^c$, so our assumptions imply that $[a,b]$ and $[a,c]$ commute. 
  This proves that
  $$[a,bc,b]=[[a,c]^b,b]=[a,c,b]^b.$$
  Finally, as $[(bc)^a,bc]=1$, using again (\ref{eq:[a,bc]}),
  $$1=[[a,bc],bc]=[[a,bc],b][[a,bc],c]^b=[a,c,b]^b[a,b,c]^b,$$
  showing that $[a,c,b]=[a,b,c]^{-1}$, as required.
\end{proof}

\begin{lem}\label{cor:x3-comm-cyc}
  If $[a,c,b]^{-1}=[a,b,c]$ and $[[a,c]^b,[a,c]]=1$, then $[a,b,c]=[c,a,b]$.
\end{lem}
\begin{proof}
  Since $[[a,c]^b,[a,c]]=1$, we have $[a,c,b]^{-1}=[[a,c],b]^{-1}
  \stackrel{\textrm{(\ref{eq:[a,b]^-1})}}{=}[[a,c]^{-1},b]=[c,a,b]$, so $[c,a,b]=[a,c,b]^{-1}=[a,b,c]$.
\end{proof}

We want to be able to apply the above lemmas where $a,b,c$ are letters, commutators or products of letters and commutators. We therefore need the following lemma:

\begin{lem}\label{lem:(a[b,c])^3}
  If $a^3=(abc)^3=(ac)^3=(acb^{-1})^3=(ac^{-1})^3=[b^a,b]=[b^c,b]=[c^b,c]=[c^a,c^b]=1$, then $(a[b,c])^3=1$.
\end{lem}
\begin{proof}
  As $[b^c,b]=[c^b,c]=1$, we have 
  \begin{equation}\label{eq:[b,c]}
  [b,c]
  \stackrel{\textrm{(\ref{eq:[a,b]^-1})}}{=}
  [b^{-1},c]^{-1}=[c,b^{-1}]
  \stackrel{\textrm{(\ref{eq:[a,b]^-1})}}{=}
  [c^{-1},b^{-1}]^{-1}=[b^{-1},c^{-1}].
  \end{equation}Also, since $(abc)^3=(ac)^3=(acb^{-1})^3=(ac^{-1})^3=1$, we have
  \begin{align}
    bcabc & =a^{-1}c^{-1}b^{-1}a^{-1}\label{eq:plc1}\\
    c^{-1}a^{-1}c^{-1} &= aca\label{eq:plc2}\\
    b^{-1}acb^{-1}a &= c^{-1}a^{-1}bc^{-1}\label{eq:plc3}\\
    ac^{-1}ac^{-1} &= ca^{-1}\label{eq:plc4}
  \end{align}
  Therefore
  \begin{align*}
    a[b,c]a[b,c] & 
    \stackrel{\textrm{(\ref{eq:[b,c]})}}{=}
    a[b^{-1},c^{-1}]a[b,c]=ab^{-1}c^{-1}bcabcb^{-1}c^{-1}=ab^{-1}c^{-1}(bcabc)b^{-1}c^{-1}=\\
    &\overset{\eqref{eq:plc1}}{=}ab^{-1}c^{-1}(a^{-1}c^{-1}b^{-1}a^{-1})b^{-1}c^{-1}= ab^{-1}(c^{-1}a^{-1}c^{-1})b^{-1}a^{-1}b^{-1}c^{-1}=\\
    &\overset{\eqref{eq:plc2}}{=} ab^{-1}(aca)b^{-1}a^{-1}b^{-1}c^{-1}=ab^{-1}ac((ab^{-1}a^{-1})b^{-1})c^{-1}=\\
    &\stackrel{\mathclap{[b^a,b]=1}}{=}\;\;\;ab^{-1}ac(b^{-1}(ab^{-1}a^{-1}))c^{-1}=a(b^{-1}acb^{-1}a)b^{-1}a^{-1}c^{-1}=\\
    &\overset{\mathclap{\eqref{eq:plc3}}}{=}\;a(c^{-1}a^{-1}bc^{-1})b^{-1}a^{-1}c^{-1}=(ac^{-1}a^{-1})(bc^{-1}b^{-1})a^{-1}c^{-1}=\\
    &\stackrel{\mathclap{[c^a,c^b]=1}}{=}\;\;\;(bc^{-1}b^{-1})(ac^{-1}a^{-1})a^{-1}c^{-1}=bc^{-1}b^{-1}ac^{-1}a^{-2}c^{-1}=\\
    &=bc^{-1}b^{-1}(ac^{-1}ac^{-1})\overset{\eqref{eq:plc4}}{=}bc^{-1}b^{-1}ca^{-1}=[b,c^{-1}]a^{-1}=\\
    &=[c^{-1},b]^{-1}a^{-1}=[c,b]a^{-1}=\left(a[b,c]\right)^{-1},
  \end{align*}
  so $(a[b,c])^3=1$.
\end{proof}

Finally, we can prove the following:

\begin{prop}\label{prop:good-imp-nilp}
  Suppose that $x,y,z,w\in G$ are good. Then $[x,y,z,w]=1$.
\end{prop}

\begin{proof}
  Our first step is to apply three times the combination of \Lref{lem:[ab,a]}, equation~\eqref{eq:[a,b]^3} and \Lref{lem:(a[b,c])^3}. Here are the exact details.
  
We first apply \Lref{lem:[ab,a]} to get that $[x^y,x]=1$ and similarly for all other disjoint simple products of letters from $\{x,y,z,w,x^{-1},y^{-1},z^{-1},w^{-1}\}$, for example $[(zy)^{xw^{-1}},zy]=1$ (recall that $x,y,z,w$ are good, namely any simple product is or order 3). In particular we get $[x^{z^{-1}y},x]=1$ which implies $[x^y,x^z]=1$ (again, for any 3 disjoint simple products). 

  These facts allow us to apply \eqref{eq:[a,b]^3} to get $[x,y]^3=1$ and \Lref{lem:(a[b,c])^3} to get $(x[y,z])^3=1$, again for any 2 and 3 disjoint simple products. These estimates we can feed back into \Lref{lem:[ab,a]} to get $[[x,y]^z,[x,y]]=[[x,y]^{zw},[x,y]]=[x^{[y,z]},x]=[x^{y[z,w]},x]=1$ (for the last case we use that if $(y^{-1}x[w,z])^3=1$ then one can conjugate with $y$ and get that $(x[w,z]y^{-1})^3=1$). Again we conclude from $[x^{y^{-1}[z,w]},x]\linebreak[4]=1$ that $[x^{[z,w]},x^y]=1$.

  We continue using \eqref{eq:[a,b]^3} to get $[[x,y],z]^3=1$ and \Lref{lem:(a[b,c])^3} to get $(w[[x,y],z])^3=1$. A final round through \Lref{lem:[ab,a]} and \eqref{eq:[a,b]^3} gives $[[[x,y],z],w]^3=1$.
  
Thus we may apply \Lref{lem:x3-comm-trans} and \Lref{cor:x3-comm-cyc} whenever $a,b,c$ are letters or commutators of $x,y,z,w$ without joint letters, as long as at most one commutator (possibly iterated) appears. On the one hand,
  \begin{align*}
    [[x,y],[z,w]] & \overset{\textrm{\ref{cor:x3-comm-cyc}}}{=} [[[z,w],x],y] = [[z,w,x],y] \overset{\textrm{\ref{cor:x3-comm-cyc}}}{=} [x,z,w,y]=\\
    & \overset{\textrm{\ref{lem:x3-comm-trans}}}{=} [x,z,y,w]^{-1} \overset{\textrm{\ref{lem:x3-comm-trans}}}{=} [[x,y,z]^{-1},w]^{-1} \overset{\textrm{\eqref{eq:[a,b]^-1}}}{=} [x,y,z,w].
  \end{align*}
  On the other hand, $[[z,w],[x,y]]=[[x,y],z,w]=[x,y,z,w]$, so $[x,y,z,w]^{-1}=[x,y,z,w]$, which can be written as $[x,y,z,w]^2=1$. As $[x,y,z,w]^3=1$, it follows that $[x,y,z,w]=1$.
\end{proof}

We can now conclude the proof of the theorem.

\begin{proof}[Proof of \Tref{thm:x3-gap}]
  Assume that $\liminf_{n\to\infty}\Pr\left(R_n^3=1\right)>1-\eps$. Take four independent random walks $X_n,Y_n,Z_n,W_n$ with respect to $S$. As previously explained, any simple product of $X_n,Y_n,Z_n,W_n$ is also distributed as a random walk on $G$ with respect to $S$, so it is of order $3$ with high probability. Therefore, the elements $X_n,Y_n,Z_n,W_n$ are good with high probability (that depends only on $\eps$).

  If $X_n,Y_n,Z_n,W_n$ are good, by \Pref{prop:good-imp-nilp} it follows that $[X_n,Y_n,Z_n,W_n]=1$. Therefore, if~$\eps$ is small enough, $\liminf_{n\to\infty}\Pr\left([X_n,Y_n,Z_n,W_n]=1\right)>0$, so by \cite[Theorem 1.5]{MTVV} the group $G$ is virtually nilpotent. By \cite[Theorem 3.25]{Hi}, any nilpotent groups is residually finite, so $G$ is also residually finite. As the law $x^3=1$ has a gap for finite groups \cite{laf}, if $\eps$ is small enough than $G$ must satisfy $x^3=1$, and we are done.
\end{proof}

\subsection{Non-Positivity result for \texorpdfstring{$x^m=1$}{x to power m}}
We now show that the law $x^m=1$ does not satisfy a positivity property. We use an example similar to the Frobenius group example which appeared in \cite{laf2}.

Let us choose a matrix $A\in \mathrm{Mat}_{m-1}(\mathbb{Z})$ such that $A^m=I$ and $\det(A-I)\neq 0$. Then we have ${(A-I)(A^{m-1}+\cdots+A+I)=0}$ and since $\det(A-I)\neq 0$ we have $A^{m-1}+ \cdots +A +I=0$. For example, one can take the companion matrix of $1+x+\dotsb+x^{m-1}$, namely the matrix $A=(a_{ij})_{i,j=1 \cdots m-1}$ with \begin{equation*}
a_{ij}=\begin{cases}
1, &\text{if }i=j+1,\\
-1, &\text{if }j=m-1,\\
0, &\text{otherwise}.
\end{cases}
\end{equation*}
Consider the semidirect product $G=\Z^{m-1}\rtimes_A \Z/m\Z$, namely, with the action of $\Z/m\Z$ given by powers of $A$.

\begin{prop} \label{non-positivity for power}
Let $G=\Z^{m-1}\rtimes_A \Z/m\Z$, and let $R_n$ be a (lazy) random walk on $G$. Then
$x^m=1$ with probability at least $\frac{\phi(m)}{n}$ with respect to the random walk~$R_n$, where $\phi(m)$ is the Euler's totient function, but $G$ does not satisfy $x^m=1$ virtually.
\end{prop}
\begin{proof}
Let us take $(v,k) \in \Z^{m-1}\rtimes_A \Z/m\Z $. Then we have\\
\begin{align*}
    (v,k)^m&=((A^{k(m-1)}+\cdots +A^{2k}+A^k+I)v,mk)\\
    &=((A^{k(m-1)}+\cdots +A^{2k}+A^k+I)v,0).
\end{align*}
If $k$ and $m$ are coprime, then the formula above can simplified and we get
$$(v,k)^m=((A^{m-1}+ \cdots + A + I)v,0)=(0,0).$$
Let $X_k=\{(v,k), \text{ where } v\in \Z^{m-1}\} \subset G$. We note that $\Pr(R_n \in X_k) \to \frac{1}{m}$ as $n \to \infty$. Then $\Pr(R_n^m=1) \geq \Pr(R_n \in \cup_{\gcd(k,m)=1} X_k) \to \frac{\phi(m)}{m}$, where $\gcd(k,m)$ is the greatest common divisor and this proves the first claim.

Now we show that $G$ does not contain a torsion finite index subgroup of exponent~$n$. First note that $X_0$ is a subgroup of $G$ of index $m$ such that for any $x\in X_0\setminus{\{0\}}$ we have $x^m \neq 1$, so $X_0$ does not satisfy the law $x^m=1$ virtually. Therefore, by \Rref{subgroup}, the group $G$ does not satisfy the law virtually.
\end{proof}

\begin{rem}\label{prime}
Let us note that we can get a higher probability by considering the group $G=\Z^{p-1}\rtimes_A \Z/p\Z$, where $p$ the largest prime factor of $m$. $G$ satisfies $x^p=1$ and hence $x^m=1$ with probability at least $1-\frac{1}{p}=\frac{\varphi(p)}{p}\geq \frac{\varphi(m)}{m}$. It follows from the proof that $G$ does not satisfy the law $x^m=1$ virtually,
\end{rem}

\section{Commutators of balanced laws}\label{sec:commut}

In the coming sections, we will be interested in laws of the form $[w_1,w_2]$. We will usually assume that $w_1,w_2$ are laws on disjoint letters. We will show that in many cases, such laws do not satisfy a gap result, and also do not satisfy a positivity result.

In fact, our results are somewhat stronger. We will show that for such laws, we can find a group~$\Gamma$, which contains a non-abelian free group as a subgroup, and satisfies the following: one can find generating sets $S$ of $\Gamma$ so that the probabilty that $\Gamma$ satisfies $[w_1,w_2]$ with respect to an $S$-random walk can be arbitrarily close to $1$.

We make the following differentiation. We call a law \textbf{balanced} if it has trivial abelianization. We will separate our discussion of commutators of laws according to whether the inner laws are balanced or not.

\begin{rem}
If $G$ has an element of infinite order and $w$ is a non-balanced law, then $G$ cannot satisfy $w$ virtually. Indeed, let $H_0\cong \Z$ be a subgroup generated by an element of infinite order and $\Lambda$ be a finite index subgroup of $G$, then $\Lambda \cap H_0$ is a non-trivial finite index subgroup of $H_0$. Let $w=w(x_1, \dots x_n)$ be a non-balanced law and the total degree of some variable $x_k$ is not zero. Let us fix some $a\neq 1$ in $\Lambda \cap H_0$. For $x_i=1$ for $i\neq k$ and $x_k=a$ we have $w(x_1, \dots x_n) \neq 1$, so any finite index subgroup $\Lambda$ does not satisfy $w$ and $G$ does not satisfy $w$ virtually.
\end{rem}

In this section we discuss the case of a commutator of two balanced laws. We first present the proof for the metabelian law $[[x,y],[z,w]]=1$; we then generalize it to an arbitrary commutator law $[w_1,w_2]=1$ (where $w_1,w_2$ are balanced).

\subsection{Word paths of random walks}\label{sub:word-paths}

\begin{defn}\label{def:path}


For a word $w\in\F_d$ and random walks $R_n^{(1)},\dots,R_n^{(d)}$, we write $\gamma=\gamma(w;R_n^{(1)},\dots,R_n^{(d)})$ to be the path defined by $w$ when substituting the random walks $R_n^{(1)},\dots,R_n^{(d)}$. Namely, if  $w=a_1^{\eps_1}\dotsb a_\ell^{\eps_l}$ where $a_j\in\{1,\dotsc,d\}$ and $\eps_j\in\{\pm 1\}$. Then $\gamma=(\gamma_0,\dotsc,\gamma_{\ell n})$ with 
\begin{equation}\label{eq:def gamma}
\gamma_{jn+k}=\begin{cases}
    \big(R_n^{(a_1)}\big)^{\eps_1}\dotsb \big(R_n^{(a_{j})}\big)^{\eps_{j+}}R_k^{(a_{j+1})} & \eps_{j+1}=1\\
    \big(R_n^{(a_1)}\big)^{\eps_1}\dotsb \big(R_n^{(a_{j})}\big)^{\eps_{j}}\big(R_n^{(a_{j+1})}\big)^{-1}R_{n-k}^{(a_{j+1})} & \eps_{j+1}=-1.
\end{cases}
\end{equation}
for all $j\in\{0,\dotsc,\ell-1\}$ and $k\in\{1,\dotsc,n\}$, and with $\gamma_0$ being the unit of the group.
\end{defn}
In words, 
the path is constructed by first following $(R_n^{(i_1)})^{\eps_1}$, then following the steps of $(R_n^{(i_2)})^{\eps_2}$ and so on. In particular, it is indeed a path in the Cayley graph, i.e.\ $\gamma_i$ and $\gamma_{i+1}$ are always neighbors.

\subsection{The metabelian law}

We turn to studying the law $[[x,y],[z,w]]=1$. Our aim is to show that this law does not satisfy a general gap result as well as any positivity result. We do this by studying the wreath product~$H\wr\Z^5$ for a non-abelian group $H$, and proving the following theorem:

\begin{thm}\label{t:meta-eps}
Let $H$ be a non-abelian group generated by two elements $a,b$, and consider the wreath product $\Gamma=H\wr\Z^5$. Then for any $\eps>0$ there exists a symmetric generating set $S$ of $\Gamma$ of size $160$ such that, if $X_n,Y_n,Z_n,W_n$ are independent simple random walks on $\Gamma$ where each step is uniform on $S$,
$$\liminf_{n\to\infty}\Pr\left([[X_n,Y_n],[Z_n,W_n]]=1\right)\geq 1-\eps.$$
In other words, $\Gamma$ satisfies the metabelian law with probability $\geq 1-\eps$ with respect to the simple random walk induced by $S$.

In fact, the generating set can be taken as follows: Let $\pm e_i$ denote the $10$ standard generators of $\Z^5$. We then take $S=\{s_1ms_2:s_1,s_2\in\{(a^{\pm 1}\delta_{\underline{0}},\underline{0}),(b^{\pm 1}\delta_{ke_1},\underline{0})\},m\in\{(\underline{\id_H},\pm e_i)\}\}$ for a large enough $k$.



\end{thm}

\begin{rems*}
    \lazyenum
        \item The same theorem also holds for random walks when taking the following $14$ generators: $\pm e_i$, $(a^{\pm 1}\delta_{\underline{0}},\underline{0})$ and $(b^{\pm 1}\delta_{ke_1},\underline{0})$. However, we work with the larger generating set to simplify the proof.
        \item A similar result also holds when $H$ is generated by more than $2$ elements. In such cases, one needs to choose a finite generating set $S$ of $H$, and for the generating set of $H\wr\Z^5$ apply each generator in a ``far away'' position from the others.
        \item One can use this result to find a sequence of probability measures $M=\set{\mu_n}$ that measures index uniformly, such that $\Pr_M\left([[x,y],[z,w]]=1\on \Gamma\right)=1$. To do this, for each $n$ one takes a generating set $S^{(n)}$ of $\Gamma$ such that one has $\Pr_{S^{(n)}}\left([[x,y],[z,w]]=1\on G\right)>1-\frac{1}{n}$. For each $n$, take $\mu_n$ to be the distribution of a random walk with step distribution uniform on $S^{(n)}$, after many steps. By taking this number of steps to be large enough, one can make sure that the supports of $\mu_n$ cover all of $G$, $M$ measures index uniformly, and $\Pr_M\left([[x,y],[z,w]]=1\on \Gamma\right)=1$.
    \end{enumerate}
\end{rems*}

Before proving the theorem, let us mention its implications. Taking $H=\F_2$ to be the free group on $2$ generators, we get that the group $F_2\wr\Z^5$ satisfies the metabelian law with probability arbitrarily close to 1. 
Because of its free subgroups it does not satisfy any law, and hence does not satisfy and law virtually. Thus the metabelian law does not satisfy neither gap nor positivity.

We turn to the proof of \Tref{t:meta-eps}. We will need the following observation on commutators in wreath products:
\begin{fact}\label{f:commutators}
Let $L_1,L_2 \in \bigoplus_{g\in G} H$. If for any $x\in G$ we have $[L_1(x),L_2(x)]=\id_H$, then in the wreath product $H\wr G$ we have $[(L_1,\id_G),(L_2,\id_G)]=\id_{H\wr G}$.
\end{fact}

We will also need several estimates on intersections of random paths in $\Z^5$, given by the following two lemmas. All random walks on $\Z^5$ are taken with respect to the standard generating set.

\begin{lem}\label{l:intersection}
Let $\gamma=(X_n)_{n=0}^{\infty}$ and $\gamma'=(Y_n)_{n=0}^\infty$ be two independent simple random walks paths on $\Z^5$ starting at $x_0$ and $y_0$ respectively. Then there is a constant $C>0$ such that
$$\Pr\left(\gamma\cap\gamma'\neq\varnothing\right) < C\cdot d(x_0,y_0)^{-\frac{1}{2}}.$$
\end{lem}
Here and below the notation $\gamma\cap\gamma'$ for two paths stands for the set of vertices appearing in both paths.
\begin{proof}
Let $r=d(x_0,y_0)$, and denote by $K=|\gamma \cap \gamma'|$ the number of intersections between the paths. Then
$$\Pr\left(\gamma \cap \gamma'\neq \varnothing\right)= \Pr(K\geq 1) \leq \E[K] \le \sum_{n=0}^\infty\sum_{m=0}^\infty \Pr(X_n=Y_m) = \sum_{m+n\geq r}\Pr(X_n=Y_m).$$
Let $S_n$ be a simple random walk on $\Z^5$. Then $Y_m-X_n-(y_0-x_0)$ has the same distribution as $S_{n+m}$. It is well known that $\Pr(S_n=x)\le Cn^{-5/2}$ for all $x\in\Z^5$. Hence
$$\Pr\left(\gamma \cap \gamma'\neq \varnothing\right)\leq\sum_{n+m\geq r}\Pr(S_{n+m}=x_0-y_0) \leq \sum_{n+m\geq r}C\frac{1}{(n+m)^{\frac{5}{2}}} \leq C'r^{-\frac{1}{2}}$$
as required.
\end{proof}

We remark that the correct asymptotic behavior is $\Pr(\gamma\cap\gamma'\neq \varnothing)\approx d(x_0,y_0)^{-1}$, but we will have no use for this extra precision, so we will not prove it.
\begin{lem}\label{l:4paths}

Let $\{X_i,Y_i,Z_i,W_i\}_{i=1}^n$ be four simple random walks on $\Z^5$.
Consider each of the random walks as a path in $\Z^5$, and denote by $\gamma,\gamma'$ the paths (of length $4n$) described by the commutators $[X_n,Y_n], [Z_n,W_n]$ respectively. Denote by $\gamma'+v$ the translation of $\gamma'$ by the vector $v$.

Then for any $\eps>0$ there exists $k_0$, independent of $n$, such that for $k\ge k_0$ we have
$$\Pr\left(\gamma \cap (\gamma'+ke_1) \neq \varnothing\right) < \eps.$$
\end{lem}
\begin{proof}
Notice that each of the commutator paths is a loop starting at the origin that can be divided into two simple random walk paths of length $2n$ each: for instance, $\gamma$ can be divided into $X_nY_n$ and $\overleftarrow{Y_n}\overleftarrow{X_n}$. The arrow indicates that the path is traversed in the opposite direction, though this will not matter to us. Therefore by symmetry
\begin{align*}
    \Pr(\gamma \cap (\gamma'+ke_1) \neq \varnothing)& \leq \Pr(X_nY_n \cap (Z_nW_n+ke_1)\neq \varnothing)+\\
    &+ \Pr(\overleftarrow{Y_n}\overleftarrow{X_n} \cap (Z_nW_n+ke_1)\neq \varnothing) + \\
    &+  \Pr(X_nY_n \cap (\overleftarrow{W_n}\overleftarrow{Z_n}+ke_1)\neq \varnothing) + \\
    &+  \Pr(\overleftarrow{Y_n}\overleftarrow{X_n} \cap (\overleftarrow{W_n}\overleftarrow{Z_n}+ke_1)\neq \varnothing) .
\end{align*}
Each of the terms can be bounded above by the intersection probability of two infinite simple random walks in $\Z^5$ starting distance $k$ apart. By \Lref{l:intersection} this is bounded by $Ck^{-\frac{1}{2}}$, so the lemma follows.
\end{proof}

We are now ready to prove the main theorem.

\begin{proof}[Proof of \Tref{t:meta-eps}]
Fix some $k\ge 0$. Recall that the generating set $S$ consists of the standard generators $\pm e_i$ of~$\Z^5$, $(a^{\pm 1}\delta_{\underline{0}},\pm e_i)$ and $(b^{\pm 1}\delta_{ke_1},\pm e_i)$. Let $X_n,Y_n,Z_n,W_n$ be independent simple random walks with step distribution uniform on $S$.

Consider the commutator paths $\gamma$ and $\gamma'$ of $[X_n,Y_n]$ and $[Z_n,W_n]$ respectively as in \Lref{l:4paths}, and write $[X_n,Y_n]=(L_1,\underline{0})$ and $[Z_n,W_n]=(L_2,\underline{0})$.
Notice that for $x\in \Z^5$, $L_1(x)=\id$ for $x\notin \gamma\cup (ke_1+\gamma)$, $L_1(x)\in A\coloneqq\{a^n\}_{n=-\infty}^{\infty}$ for $x\in \gamma\setminus (ke_1 + \gamma)$ and $L_1(x)\in B\coloneqq\{b^n\}_{n=-\infty}^{\infty}$ for $x\in (ke_1 + \gamma)\setminus \gamma$.

Fix $\eps>0$. By \Lref{l:4paths} we can choose $k$ large enough so that with probability $\geq 1-\eps$ the commutator loops satisfy $\gamma \cap (ke_1+\gamma')=\varnothing$ and  $(\gamma+ke_1)\cap \gamma'=\varnothing$. Under this event, any $x\in\Z^5$ falls into at least one of the following cases:
\begin{enumerate}
    \item $x\notin \gamma\cup (ke_1+\gamma)$. In this case $L_1(x)=\id_H$ hence $[L_1(x),L_2(x)]=\id_H$.
    \item $x\notin \gamma'\cup (ke_1+\gamma')$. In this case $L_2(x)=\id_H$ hence $[L_1(x),L_2(x)]=\id_H$.
    \item $x\in (\gamma\cap \gamma') \setminus ((ke_1+\gamma)\cup (ke_1+\gamma'))$. In this case  $L_1(x),L_2(x)\in A$ and thus $[L_1(x),L_2(x)]=\id_H$.
    \item $x\in ((ke_1+\gamma)\cap (ke_1+\gamma')) \setminus (\gamma \cup \gamma')$. In this case  $L_1(x),L_2(x)\in B$ and thus $[L_1(x),L_2(x)]=\id_H$.
\end{enumerate}
It now follows from \Fref{f:commutators} that $[[X_n,Y_n],[Z_n,W_n]]$ is trivial. Therefore
$$\Pr\left([[X_n,Y_n],[Z_n,W_n]]=1\right)\geq \Pr\left(\gamma \cap (ke_1+\gamma')=(\gamma+ke_1)\cap \gamma'=\varnothing\right)\ge 1-\eps$$
as required.
\end{proof}

\subsection{The general case}

The ideas of the previous theorem can be generalized to give the following result:

\begin{thm}\label{t:bal-comm-eps}
Let $w_1\in\F_d$ and $w_2\in\F_{d'}$ be two balanced laws such that $[w_1,w_2]$ is not trivial. Let $H$ be a non-abelian group generated by two elements $a,b$, and consider the wreath product $\Gamma=H\wr\Z^5$. Then for any $\eps>0$ there exists a symmetric generating set $S$ of $\Gamma$ of size $160$ such that, if $R_n^{(1)},\dots,R_n^{(d)},S_n^{(1)},\dots,S_n^{(d')}$ are independent simple random walks on $\Gamma$ where each step is uniform on $S$,
$$\liminf_{n\to\infty}\Pr\left([w_1(R_n^{(1)},\dots,R_n^{(d)}),w_2(S_n^{(1)},\dots,S_n^{(d')})]=1\right)\geq 1-\eps.$$
In fact, the generating set $S$ can be taken as in \Tref{t:meta-eps}.
\end{thm}

In particular, for any two balanced laws $w_1$ and $w_2$ 
the law $[w_1,w_2]$ does not satisfy a gap or a positivity result.

The proof of this theorem will be similar to the proof of \Tref{t:meta-eps}.
We consider again the group $\Gamma=H\wr\Z^5$ with the same generating set as above.


We first give an analog of \Lref{l:intersection} on the intersections of such paths. Note that these are no longer random walk paths, but can still be decomposed into segments of random walks.

\begin{lem}\label{l:gen-intersection}
  Let  $R_n^{(1)},\dots,R_n^{(d)}$ and  $S_n^{(1)},\dots,S_n^{(d')}$ be $d+d'$ independent random walks on $\Z^5$ of length~$n$. Set $\gamma=\gamma(w_1,R_n^{(1)},\dots,R_n^{(d)})$ and $\gamma'=\gamma(w_2,S_n^{(1)},\dots,S_n^{(d')})$ to be the corresponding paths.
    Then for any $\eps>0$ there exists $k_0$, independent of $n$, such that for $k\ge k_0$ we have
$$\Pr\left(\gamma \cap (\gamma'+ke_1) \neq \varnothing\right) < \eps.$$
\end{lem}
\begin{proof}
  Denote by $\gamma_i$ and $\gamma'_i$ the $i$'th step in $\gamma$ and in $\gamma'$. We begin as in \Lref{l:intersection}: Denote by $K=|\gamma \cap (\gamma'+ke_1)|$ the number of intersections between the paths. Then
  \begin{align}
    \Pr\left(\gamma \cap (\gamma'+ke_1)\neq \varnothing\right)&= \Pr(K\geq 1) \leq \E[K] \le \sum_{i=0}^{|\gamma|}\sum_{i'=0}^{|\gamma'|} \Pr(\gamma_i=\gamma'_{i'}+ke_1) =\nonumber\\ 
    &=\sum_{i=0}^{|\gamma|}\sum_{i'=0}^{|\gamma'|} \Pr(\gamma_i-\gamma'_{i'}=ke_1).\label{eq:sum gammaiip}
  \end{align}
  Examining the definition of $\gamma$ \eqref{eq:def gamma} we see that $\gamma_i$ is the sum of several copies of various $R_n^{(j)}$, some with positive sign and some with negative sign, as well as one copy of $R_k^{(j_0)}$ for some $j_0$ (coming from the random walk path on which the $i$'th step lies). We wish to write $\gamma$ as a combination of \emph{independent} random walks, so we write $R_n^{(j_0)}=R_k^{(j_0)}+(R_n^{(j_0)}-R_k^{(j_0)})$ and note that $R_n^{(j_0)}-R_k^{(j_0)}$ is itself a random walk, and is independent of $R_k^{(j_0)}$ (and, of course, of the other $R_n^{(j)}$). We have thus arrived at a representation
  $$
  \gamma_i=\sum_{j=1}^{d+1}b_j X_{m_j}^{(j)}
  $$
  where $b_j\in\Z$, $m_j=n$ except, possibly, for two values of $j$, and $X^{(j)}$ are independent random walks. Further, each value of $m_j$ different from $n$ appears for at most $2|w|$ different values of $i$, and each $|b_j|\leq|w_1|$.

  Repeating this calculation for $\gamma'$ we may write
  $$\gamma_i-\gamma'_{i'} = \sum_{j=1}^{d+d'+2}b_j X_{m_j}^{(j)}.$$
  Rearrange the nonzero terms $X^{(j)}$ so that $m_1\ge m_2 \ge \dotsb$. 
  Let $l=\max\set{|w_1|,|w_2|}$. As $|b_j|\leq l$ for all~$j$, we get that if $m_1<\frac{k}{l(d+d'+2)}=c'k$, then $|\gamma_i-\gamma'_{i'}|<k$, and cannot be equal to $ke_1$. 
  Note that after rearranging, the number of $(i,i')$ for which $m_1=M$ is at most $2Ml^2$, for any $M<n$. For $M=n$ it will be enough to use the trivial bound that the number of couples $(i,i')$ with $m_1=n$ is smaller than the total number of couples, $|w_1||w_2|n^2$.

 Condition on the values of $X_{m_2}^{(2)},\dots,X_{m_{d+d'+2}}^{(d+d'+2)}$. We get
\begin{align*}
\Pr(\gamma_i-\gamma'_{i'}=ke_1)
&\le \max_{v\in\Z^d}\Pr\bigg(\gamma_i-\gamma'_{i'}=ke_1 \,\bigg|\, \sum_{j=2}^{d+d'+2}b_jX_{m_j}^{(j)}=v\bigg)\\
&=\max_{v\in\Z^d}\Pr\big(b_1X_{m_1}^{(1)}=ke_1-v\big)\le Cm_1^{-5/2}.
\end{align*}
  Using the multiplicity of possible values for $m_1$ explained above and the fact that if $m_1<c'k$ then this probability is zero allows to estimate the sum in \eqref{eq:sum gammaiip},
  \begin{align*}
  \sum_{i=0}^{|\gamma|}\sum_{i'=0}^{|\gamma'|} \Pr(\gamma_i-\gamma'_{i'}=ke_1)
  &= \sum_{M=c'k}^n \sum_{(i,i'):m_1=M}\Pr(\gamma_i-\gamma'_{i'}=ke_1)\\
  &\le \sum_{M=c'k}^{n-1} 2Ml^2\cdot CM^{-5/2} + l^2n^2\cdot Cn^{-5/2}\le Ck^{-1/2}
\end{align*}
(the last constant may depend on $w_1$ and $w_2$). With \eqref{eq:sum gammaiip} the lemma is proved.
%
\end{proof}

We can now prove the theorem.

\begin{proof}[Proof of \Tref{t:bal-comm-eps}]
    This is essentially the same proof as \Tref{t:meta-eps}. Since tha base group $\Z^5$ is abelian and the laws $w_1,w_2$ are balanced, each one of them defines a loop. By \Lref{l:gen-intersection}, the probability that the $w_1$-loop and the translation of the $w_2$-loop are disjoint (and vice versa) can be made arbitrarily close to $1$ by choosing $k$ large enough. The rest of the proof remains verbatim as the proof of the metabelian case.
\end{proof}


\section{Commutators of a power and a balanced law}\label{sec:comm-power-balanced}

Let us take a matrix $A_0=(a_{ij})_{1\le i,j\le m}$ of a cyclic permutation, that is  \begin{equation*}
    a_{ij}=\begin{cases}
    1 & \text{ if}\quad j-i \equiv 1  \pmod{m},\\
    0 &\text{otherwise}.
    \end{cases}
\end{equation*}
For any $d$ that divides $m$, let us consider the subspace $V_d\subset \mathbb{C}^m$ defined by
$$V_d=\set{(x_1, \cdots x_m)\in\mathbb{C}^m\suchthat \forall 1\leq i \leq \frac{m}{d}: \quad x_i+x_{i+\frac{m}{d}}+ \cdots +x_{i+\frac{(d-1)m}{d}}=0}.$$
We have $\dim V_d = m-\frac{m}{d}$, and each $V_d$ is $A_0$-invariant. Let us fix a primitive $m$-th root of unity $\xi$. The vectors $v_k:=(1,\xi^k,\xi^{2k}, \cdots , \xi^{(m-1)k})$ are eigenvectors of $A_0$, i.e.\ $A_0v_k=\xi^kv_k$. Note also that $v_k \in V_d$ if and only if $d$ does not divide $k$, so the number of such $v_k$ is exactly $m-\frac{m}{d}$, and they form a basis of $V_d$.

Let us define the space $V=\cap_{1\neq d|m} V_d$. By the construction, $V$ has the basis $\set{v_k}_{\gcd(k,m)=1}$, hence $\dim V=\varphi(m)$, where $\varphi$ is the Euler's totient function, and the eigenvalues of $A_0|_V$ are exactly the primitive $m$-th roots of unity $\xi^k$ with $\gcd(k,m)=1$. The space $V$ is defined by rational equations, so $V \cap \Z^m \cong \Z^{\varphi(m)}$. Moreover, $A_0\cdot V=V$ and $A_0 \cdot \Z^m=\Z^m$, hence $A_0\cdot(V \cap \Z^m)=V \cap \Z^m.$

Fix a $\Z$-basis $e_i$, $i=1,\ldots,\varphi(m)$ of $V\cap\mathbb{Z}^m$. For convenience, we identify $V\cap\Z^m$ with $\Z^{\varphi(m)}$, by identifying the $\Z$-basis $e_i$ with the standard generators of~$\Z^{\varphi(m)}$. Let~$A$ be the matrix of $A_0|_{V}$ in the basis $\set{e_i}$. Observe that since $A_0^m=1$, we have $A^m=1$.

\begin{lem}\label{lem:semi}
Let $A$ be the matrix defined above, and let $f(x)\in\mathbb{Q}[x]$. Then for the matrix $M=f(A)$ we have either $M=0$ or $\det(M)\neq 0.$
\end{lem}
\begin{proof}
As $A$ is diagonalizable with eigenvalues $\xi^k$ for any $\gcd(k,m)=1$, the matrix $f(A)$ is diagonalizable with eigenvalues $f(\xi^k)$ for such $k$. If $f(\xi^k)=0$ for some~$k$ with $\gcd(k,m)=1$, then $f$ is divisible by the cyclotomic polynomial $\Phi_m(x)$, and thus $f(\xi^j)=0$ for any $j$ with $\gcd(m,j)=0$, i.e.\ $M=f(A)=0$. Otherwise, $M$ has no zero eigenvalues, so $\det(M)\neq 0$.
\end{proof}

Let us consider the group $G_m=\mathbb{Z}^{\varphi(m)}\rtimes_A\mathbb{Z}/m\mathbb{Z}$, where the semidirect product is taken using the matrix $A$ defined as above. Consider a simple random walk $R_n$ on $G_m$ corresponding to the generating set $(\pm e_i,k), (0, k)$ where $i=1,\ldots,\varphi(m)$, $k=0,\ldots, m-1 \in \mathbb{Z}/m\mathbb{Z}$.

\begin{thm}\label{power-balanced}
Let $H$ be a non-abelian group generated by two elements $a$ and $b$ and consider the wreath product $\Gamma=H \wr G_m^5$. Then for any $\varepsilon>0$ and any balanced law $w=w(y_1,\dots, y_d)$ there exists a symmetric generating set $S$ on $\Gamma$ such that if $X_n$, $Y^{(1)}_n, \dots, Y^{(d)}_n$ are independent simple random walks of $\Gamma$ where each step is uniform on $S$, then
$$\liminf_{n\to\infty}\Pr\left([(X_n)^m,w(Y^{(1)}_n, \dots, Y^{(d)}_n)]=1\right)\geq \left(\frac{\varphi(m)}{m^{d+1}}\right)^5-\eps.$$
In other words, $\Gamma$ satisfies the law $[x^m,w]$ with positive probability with respect to the simple random walk induced by $S$.
\end{thm}

As before, the theorem holds one $H$ is taken to be the free group on two generators, and in this case, since this group contains a free subgroup it cannot satisfy any law virtually, so we conclude that $[x^m,w(y_1,\dotsc,y_d)]$ does not satisfy positivity.

To prove this theorem we will use a strategy similar to the proof of \Tref{t:meta-eps}. In particular, recall the definition of the path of a word
(\Dref{def:path}). Let $R_n^{(1)},\dotsc,R_n^{(d+1)}$ be independent random walks on $G_m^5$ and let $\gamma$ and $\gamma'$ be the paths of $\big(R_n^{(d+1)}\big)^m$ and $w(R_n^{(1)},\dotsc,R_n^{(d)})$. Our goal is to prove that $\Pr(\gamma\cap\gamma'q\neq \varnothing )$ is small for an appropriate $q\in G_m^5$.

First, we want to simplify the statement. We will assume $q=(q_0,\dotsc,q_0)$ for some $q_0\in G_m$ and then
$$\Pr(\gamma\cap\gamma'q\neq \varnothing )\le 
\sum_{\substack{k\le mn\\ l\le n|w|}}\Pr((\gamma'_{l,1})^{-1}\gamma_{k,1}=q_0)^5,$$
where $\gamma_l'$ is the $l$-th step of $\gamma'$, $\gamma_k$ is the $k$-th step of $x^m$ and also $\gamma_{l,1}'$ and $\gamma_{k,1}$ are the first ``coordinates" of $\gamma_l'$ and $\gamma_k$ respectively.
As in the proof of \Lref{l:gen-intersection}, we may write
$$(\gamma_{l,1}')^{-1}\gamma_{k,1}=\widehat{w}_{k,l}(R_r^{(1)},R_{n-r}^{(2)},R_{r'}^{(3)},R_{n-r'}^{(4)}, R_n^{(5)},\ldots R_n^{(d+3)})$$
where $R_r^{(1)},R_{n-r}^{(2)},R_{r'}^{(3)},R_{n-r'}^{(4)}, R_n^{(5)},\ldots R_n^{(d+3)}$ are independent random walks and where $r$ depends on $k$, and $r'$ on $l$. Hence
$$\Pr(\gamma\cap\gamma'q\neq \varnothing )\le 
\sum_{\substack{k\le nm\\ l\le n|w|}} \left(\Pr(\widehat{w}_{k,l}(R_r^{(1)},R_{n-r}^{(2)},R_{r'}^{(3)},R_{n-r'}^{(4)}, R_n^{(5)},\ldots R_n^{(d+3)})=q)\right)^5.$$
Therefore the proof of \Tref{power-balanced} boils down to getting an upper bound on the probability that a word evaluated at some random walks will hit a far-away point. This was done in \Sref{sec:commut} for abelian groups. Here $G_m$ is not abelian, however we can exploit its structure to get a similar result. 
\begin{lem}
    For any fixed word $\widehat{w}=\widehat{w}(z_1,\ldots, z_u)$ there is a finite set $B\subset \mathrm{Mat}_{\varphi(m)}(\mathbb{Z})$ such that for any $(v_i,k_i)\in G_m$ we have
    $$\widehat{w}((v_1,k_1),\ldots, (v_u,k_u))=\left(\sum_{i=1}^{u} b_i\cdot v_i, * \right)$$
    for some $b_i\in B$, where $*$ is some element of $\mathbb{Z}/m\mathbb{Z}$.
\end{lem}
\begin{proof}
    Multiplication and taking the inverse in $G_m$ is given by multiplication of the first coordinate by the powers of $A$. So we can take as $B$ a set of all sums of $\pm A^k$ of length at most $|\widehat{w}|$.
\end{proof}

Let $q=(v_0,k_0)\in G_m$, and let $R_*^i=(v_i,k_i)$ be random walks on $G_m$. Write $B=B(\widehat{w})$ and $\widehat{w}_{k,l}(R_r^{(1)},R_{n-r}^{(2)},R_{r'}^{(3)},R_{n-r'}^{(4)}, R_n^{(5)},\ldots ,R_n^{(d+3)})=\left(\sum_{i=1}^{d+3} b_i\cdot v_i, * \right)$ by the previous lemma. The $b_i$'s only depend on $k_1,\dotsc,k_{d+3}$, but actually we do not need to calculate them. Instead we can use the fact that $B$ is finite together with a union bound. 
Thus we need to bound the probabilities $\Pr\left(\sum_i b_i v_i = v_0\right).$
%

\begin{lem}
    For any $n>0$ and any $b\in B\setminus\{0\}$ if $R_n$ is random walk on $G_m$ and if $R_n=(v,k)$ then
    $$\Pr(b\cdot v=v_0)\le Cn^{-\varphi(m)/2} e^{-c|v_0|^2/n}.$$
\end{lem}
\begin{proof}
    By \Lref{lem:semi} for $b\in B\setminus\{0\}$ we have $\det(b)\ne 0$. Now $\Pr(b\cdot v=v_0)=\Pr(v=b^{-1}\cdot v_0)$ and $|b^{-1}\cdot v_0|\ge c|v_0|$ with $c=1/\max_{b\in B\setminus \{0\}}||b||$. The lemma then follows from the usual heat kernel bound for random walks in $\Z^{\varphi(m)}$.
\end{proof}

\begin{lem}
    For any fixed lengths of random walks $R^{(i)}_*$ and any $b_1,\dots,b_{d+3}\in B$, denoting $R_*^{(i)}=(v_i,k_i)$ we have:
    $$\Pr\left(\sum_i b_i\cdot v_i=v_0\right)\le Cs^{-\varphi(m)/2} e^{-c|v_0|^2/s}.$$
where $s=\sum_{i:b_i\neq 0} \mathrm{len}(R^{(i)})$.
\end{lem}
\begin{proof}
    Direct computations show that convolution of two bounds from the lemma above with $n=n_1$ and $n=n_2$ is the function of the same type with $n=n_1+n_2$ and maybe different constants.
\end{proof}

We now return to $\widehat{w}_{k,l}=\widehat{w}_{k,l}(R_r^{(1)},R_{n-r}^{(2)},R_{r'}^{(3)},R_{n-r'}^{(4)}, R_n^{(5)}\ldots ,R_n^{(d+3)})$. We wish to simply sum over the finite set of possibilities for the $b_i$, but we need to restrict the set slightly. We therefore note that $b_1$ and $b_2$ cannot both be zero. This is because $R_r^{(1)}$ and $R_{n-r}^{(2)}$ appear in $\widehat{w}_{k,l}$ together, except the very last letter in $\widehat{w}_{k,l}$ which is $R_r^{(1)}$. The same holds for $b_3$ and $b_4$. We get that $s\ge \min\{r,n-r\}+\min\{r',n-r'\}$. Denote this value by $s'$. 
Since $s^{-\frac{\varphi(m)}{2}}e^{\frac{-|v_0|^2}{s}}$ as a function of $s$ takes a maximum value at $s=\frac{2|v_0|^2}{\varphi(m)}$, we get
$$
s^{-\varphi(m)/2}e^{-|v_0|^2/s}\le\max_{s\ge s'}s^{-\varphi(m)/2}e^{-|v_0|^2/s}\le C\max\{s',|v_o|^2\}^{-\varphi(m)/2}.
$$
This allows to write
\begin{align*}\Pr(\widehat{w}_{k,l}=q)
&\le \sum_{\substack{b_1,\dotsc,b_{d+3}\in B\\b_1\ne 0\textrm{ or }b_2\ne 0\\b_3\ne 0\textrm{ or }b_4\ne 0}}\Pr\bigg(\sum_jb_jv_j=v_0\bigg)\\
&\le \sum_{b_1,\dotsc,b_{d+3}\in B} C\max\{s',|v_0|^2\}^{-5\varphi(m)/2}\le C\max\{s',|v_0|^2\}^{-5/2}
\end{align*}
for $q=(q_0,\dotsc,q_0)$ and $q_0=(v_0,k_0)$. We sum these over $k$ and $l$ and use the fact that $s'$ (which is a function of $k$ and $l$) takes any value $S$ no more than $4m|w|(S+1)$ times. Hence
\begin{align*}
\sum_{k\le mn}\sum_{l\le n|w|}\Pr(\widehat{w}_{k,l}=q)
&\le\sum_{S=1}^\infty\sum_{(k,l):s'=S}C\max\{S,|v_0|^2\}^{-5/2}\\
&\le \sum_{S=1}^\infty 4m|w|(S+1)\cdot C\max\{S,|v_0|^2\}^{-5/2}
\le \frac{C}{|v_0|}.
\end{align*}
%
%
%
Thus we proved the the following statement.
\begin{prop}\label{prop:Gm-loops-prob}
For $q=(q_0,\ldots,q_0)$, $q_0\in G_m$, $q_0=(v_0,k_0)$ we have
$$\Pr(\gamma\cap\gamma'q\neq \varnothing ) < C|v_0|^{-1}.$$
\end{prop}
We are now ready to prove \Tref{power-balanced}
\begin{proof}
Let us take $q=(q_0,\dots q_0)\in G_m^5,$ where $q_0=(v_0,k_0)\in G_m$ with large enough $|v_0|$. We will choose the following generating set of $\Gamma=H \wr G_m$
$$
\{s_1ms_2:s_1,s_2\in\{a^{\pm1}\delta_0,b^{\pm1}\delta_q\},m\in\{(\pm e_i,k)\}_{i,k}^5\}.
$$
In the formula above, $a^{\pm1}\delta_0$ and $b^{\pm1}\delta_q$ are both elements of $\oplus_{G_m}H$ (which is embedded in $\Gamma$) while $(\pm e_i,k)$ are elements of $G_m$, and $G_m^5$ is also embedded in $\Gamma$ (the parameter $i$ takes values in $\{1,\dotsc,5\}$ and $k$ in $\{1,\dotsc,m\}$). 
Note that the projection of the $n$-th step of the simple random walk on $\Gamma$ corresponding to this generating set is the $n$-th step of the simple random walk that corresponds to $S^5$ of $G_m^5$.

We will use the same idea as in the previous section:  $(X_n)^m$ and $w(Y^{(1)}_n, \dots, Y^{(d)}_n)$ commute if the corresponding paths $\gamma$ and $\gamma'$ in the base group are loops and $\gamma\cap\gamma'q=\varnothing$, $\gamma q\cap\gamma'=\varnothing$. Similar to the proof of \Pref{non-positivity for power} we get that $\gamma$ is a loop with probability $\left(\frac{\varphi(m)}{m}\right)^5$, and $\gamma'$ is a loop if all $Y^{(i)}_n\in H\wr \mathbb{Z}^{5\varphi(m)} < \Gamma$, which happens independently for each $i$ with probability $\left(\frac{1}{m}\right)^5$.
And finally, by \Pref{prop:Gm-loops-prob} we have
$$\Pr(\gamma\cap\gamma'q\ne \varnothing)+\Pr(\gamma q\cap\gamma'\ne \varnothing)\le \varepsilon$$
for large enough $|q_0|$. Therefore
$$\liminf_{n\to\infty}\Pr\left([(X_n)^m,w(Y^{(1)}_n, \dots, Y^{(d)}_n)]=1\right)\geq \left(\frac{\varphi(m)}{m^{d+1}}\right)^5-\eps.$$
as required.
\end{proof}

\begin{rem*} The theorem can also be proved as a corollary of \Lref{lem:non-peano is enough} below (at the expense of replacing the power 5 with 8), but the resulting simplification is not particularly significant.
\end{rem*}

\section{Commutators of a power and a non-balanced law}\label{sec:comm-power-not-balanced}

Let us consider the general Heisenberg group of dimension $2m-1$
$$H_{2m-1}=\set{\left(\begin{array}{c|ccc|c}1&-&u^t&-&a\\\hline0&1&&0&\mid\\\vdots&&\ddots&&v\\0&0&&1&\mid\\\hline 0&0&\cdots&0&1\end{array}\right)\suchthat u,v\in\Z^{m-1},a\in\Z}.$$
For convenience we will write elements of $H_{2m-1}$ as triples $(u, a, v)$, where ${u, v\in \Z^{m-1}, a\in \Z}$. Then multiplication is given by the formula $$(u_1,a_1,v_1)\cdot (u_2, a_2, v_2)=(u_1+u_2, a_1+a_2+u^t_1v_2, v_1+v_2).$$
Let us take a matrix $A$ as in \Sref{sec:comm-power-balanced}. Define an action of $\Z/m\Z$ on $H_{2m-1}$ by
$$1*\left(\begin{array}{c|ccc|c}1&-&u&-&a\\\hline0&1&&0&\mid\\\vdots&&\ddots&&v\\0&0&&1&\mid\\\hline 0&0&\cdots&0&1\end{array}\right)= \left(\begin{array}{c|ccc|c}1&-&u^tA&-&a\\\hline0&1&&0&\mid\\\vdots&&\ddots&&A^{-1}v\\0&0&&1&\mid\\\hline 0&0&\cdots&0&1\end{array}\right)$$
This is indeed an action of $\Z/m\Z$ on $H_{2m-1}$ since it preserves multiplication and $A^m=1$.

Let us consider the semidirect product $H_{2m-1}\rtimes_A\Z/m\Z$.

\begin{prop}
Let $G=H_{2m-1}\rtimes_A\Z/m\Z$ and let $\mu$ be any finitely supported symmetric generating measure of $G$ with $\mu(1)>0$, and let $R_n$ be a random walk with step distribution $\mu$. Then for any law $w(x, y_1, \dots, y_d)$ with a total degree of some $y_i$ not equal to zero we have $[x^m,w]=1$ with probability at least $\frac{\phi(m)}{m}$ with respect to the random walk $R_n$, where $\phi(m)$ is the Euler's totient function, but G does not satisfy $[x^m,w]=1$ virtually.
\end{prop}
Note that, unlike our other results on commutators, that all require the two words to have different letters, in this result the letter $x$ might be common to both terms. The unbalanced letter, though, must be one of the $y_i$.
\begin{proof}
Let us take $x=\left( \left(u,a,v\right),k \right)\in G$. We have $x^m=((u',a',v'),km)$,  where
\begin{align*}
    (u')^t&=u^t\big(I+A^k + \dots +A^{k(m-1)}\big),\\
a'&= ma+u\big((m-1)I+(m-2)A^k+ \dots +A^{(m-2)k}\big)v,\\
v'&=\big(I+A^{-k} \dots +A^{-k(m-1)}\big)v.
\end{align*}
If $k$ and $m$ are coprime, then $u'=v'=0$, so $x^m$ lies in the center of $G$ and $[x^m,w]=1$ for any $y_1, \dots y_l\in G$.
Let $R^x_n$,$R^{y_1}_n, \dots, R^{y_d}_n $ be the set of independent random walks. Let
$X_k=\{((u,a,v),k), \text{ where } (u,a,v)\in H\} \subset G$. We note that $\Pr(R^x_n \in X_k) \to \frac{1}{m}$ as $n \to \infty$. Then
$$\Pr([(R_n^x)^m, w(R_n^{x},R_n^{y_1}, \dots, R_n^{y_d})]=1) \geq \Pr(R_n^x \in \cup_{\gcd(k,m)=1} X_k) \to \frac{\phi(m)}{m},$$ where $\gcd(k,m)$ is the greatest common divisor and this proves the first claim.

Now let us show that $G$ does not contain a finite index subgroup which satisfies the given law. Let us assume that there is a finite index subgroup $\Lambda$ of $G$ such that for any $x, y_1 \dots y_d \in \Lambda$ one has $[x^m,w(y_1, \dots, y_d)]=1$. Let $\{e_1, \dots, e_{m-1}\}$ be the standard generating set of $\mathbb{Z}^{m-1}$. Let us consider the subgroups $K_u<H<G$ and $K_v<H<G$ defined as follows:
$$K_u=\{(u,a,v)\mid u=\lambda e_1, a=0, v=0, \lambda \in \Z\}\cong \Z, $$
$$K_v=\{(u,a,v)\mid u=0, a=0, v=\lambda e_1, \lambda \in \Z\}\cong \Z. $$
Note that $\Lambda \cap K_u$ and $\Lambda\cap K_v$ are finite index subgroups of $K_u$ and $K_v$ respectively, and hence are nontrivial. Let us take $g\in (\Lambda\cap K_u)\setminus\{0\}$ and $h\in (\Lambda\cap K_v)\setminus\{0\}$. Let us fix an index $i$ such that $w(x, y_1, \dots y_d)$ has a total degree of $y_i$ not equal to zero. Then for $x=g,y_i=h$ and $y_j=0$ for $j \neq i$ we have $x^m=(\lambda e_1, 0,0)$, $w=(u,a,\mu e_1)$, where $\lambda$ and $\mu$ are non-zero integers. Hence, $[x^m, w(x, y_1, \cdots, y_l)]=(0,\lambda\mu,0)\neq 1.$ Thus we get elements $x,y_1, \dots, y_k \in \Lambda$ that do not satisfy $[x^m,w]=1$. A contradiction.
\end{proof}
\begin{rem}
Similar to \Rref{prime} one can sometimes increase the probability by considering $G=H_{2p-1}\rtimes_A \Z/p\Z$, where $p$ the largest prime factor of $m$.
\end{rem}

\section{Commutators of a law with itself}\label{sec:self-comm}

Let $w\in\F_d$ be a word, and consider the law $[w,w]\in\F_{2d}$ given by
$$[w,w](x_1,\dots,x_d,y_1,\dots,y_d)=[w(x_1,\dots,x_d),w(y_1,\dots,y_d)].$$
Our goal is to show that if $w$ does not satisfy a gap or a positivity result, the same holds for $[w,w]$. The case where $w$ is balanced is covered by \Sref{sec:commut}, so we focus on the case where $w$ is not balanced.\medskip

\begin{prop}
Let $R_n$ be a random walk on $H \wr G$ with a step distribution $\mu$. Let $R_n'$ be a random walk on $G$ with a step distribution $\pi_*(\mu)$ where $\pi$ is a natural projection $\pi: H \wr G \to G$. If $G$ satisfies a law $w$ with probability $1-\eps$ with respect to $R_n'$, and $H$ is any abelian group, then $H\wr G$ satisfies $[w,w]$ with probability at least $(1-\eps)^2$.
\end{prop}

\begin{proof}
Since $H$ is abelian, we know that $\bigoplus_{i=1}^{\infty}H$ is also abelian and we have
\begin{align*}
  &\Pr([w(R_n^{x_1}, \dots R_n^{x_d}), w(R_n^{y_1}, \dots, R_n^{y_d})]=1) \geq\\
  &\geq \Pr\left(w(R_n^{x_1}, \dots R_n^{x_d}), w(R_n^{y_1}, \dots, R_n^{y_d}) \in \ker \pi \cong\bigoplus_{i=1}^{\infty}H\right)=\\
  &=(\Pr(w(\pi(R_{n}^{x_1}), \dots \pi(R_n^{x_d}))=1\in G))^2.
\end{align*}
Hence, $H \wr G$ satisfies the law $[w,w]$ with a probability at least $(1-\eps)^2.$
\end{proof}

\begin{prop}
    Suppose that an infinite $G$ does not satisfy a law $w=w(x_1,\dots,x_r)$ virtually, and that $H$ is a group that does not satisfy $w$. Then $H\wr G$ does not satisfy $[w,w]$ virtually.
\end{prop}

\begin{proof}
    Let $\Lambda$ be a finite index subgroup of $\Gamma=H\wr G$. We want to show that $\Lambda$ does not satisfy the law $[w,w]$.

Let $\Gamma_0=\set{(\boldsymbol{1},g)\suchthat g\in G}\cong G$, i.e.\ all elements of $\Gamma$ with trivial lamps. Since $[\Gamma_0:\Gamma_0\cap \Lambda]<\infty$ and  $\Gamma_0\cong G$, it has no subgroups of finite index which satisfy $w$, so there are elements $t_1,\dots,t_r\in \Lambda \cap \Gamma_0$ such that $w(t_1,\dots,t_r)= (\boldsymbol{1},g_0)\neq 1$.
Let $g_1,g_2,\dots\in G$ such that
$$\left|g_j^{-1}g_i\right|\ge 2|g_0|$$
for all $i,j$. Consider the subgroup $\Gamma_1=\set{(L,e)\suchthat \supp(L)\sub\set{g_1,g_2,\dots}}\cong\bigoplus_{i=1}^{\infty}H$ which does not satisfy $w$ virtually. Again  we have $[\Gamma_1:\Gamma_1\cap \Lambda]<\infty$; therefore there are elements $x_1,\dots,x_r\in\Gamma_1\cap \Lambda$ such that $w(x_1,\dots,x_r)=(L_0,e)\neq 1$.
Finally, we have $w(t_1,\dots,t_r)\cdot w(x_1,\dots,x_r)=(L_1,g_0)$ and $w(x_1,\dots,x_r)\cdot w(t_1,\dots,t_r)=(L_0,g_0)$ where $\supp(L_1)\sub\set{g_1g_0,g_2g_0,\dots}$ which does not intersect $\supp(L_0)$, so $L_1\ne L_0$ and hence $[w(t_1,\dots,t_r), w(x_1,\dots,x_r)]\ne 1$.
\end{proof}

\begin{cor}
~
\begin{enumerate}
    \item If $w$ does not satisfy a (strong) gap, neither does $[w,w]$.
    \item If $w$ does not satisfy a (strong) positivity result, neither does $[w,w]$.
\end{enumerate}
\end{cor}


\section{Occupation measure of random walk paths in balls}\label{sec:occupation-measure}

The aim of this section is to provide upper bounds on the occupation measure of a path $\gamma(w;R_n^{(1)},\dots,R_n^{(d)})$ in balls (of the Cayley graph of $G$). These results will be used in the following section, and were put in a separate section since they may be of independent interest.

We start with two lemmas that estimate conditioned random walks
for transient groups, with the eventual aim of proving \Lref{lem:LPSZ for paths}.
\begin{lem}
\label{lem:superpoly}Let $G$ be a finitely generated group with
superpolynomial growth and let $S$ be some set of generators. Let
$\set{R_t}_{t\leq n}$ be a random walk on the group $G$ with respect to $S$. Let $r\in\mathbb{N}$. Then
\[
\E\big(\max_{x\in G}\E\big(|\{t<n:R_{t}\in B(x,r)\}|\,\big|\,R_{n}\big)\big)\le Cr^{3}.
\]
The constant $C$ may depend on $G$ and on $S$, but not on $r$
and $n$.
\end{lem}

In words, the typical number of visits to a ball of radius $r$ is
not more than $Cr^{3}$, even when one conditions on $R_{n}$, the
place where the random walk ``ends'', as long as~$R_{n}$ is typical.
Let us remark that this does not hold for every value of $R_{n}$.
For example, a random walk on the free group, conditioned on $R_{n}=1$
i.e.\ on returning to the identity, visits the identity about $\sqrt{n}$
times.

It is reasonable to conjecture that the value $r^{3}$ can be replaced
by $r^{2}$, as is conjectured (and still open) for the number of
visits to $B(1,r)$ for an unconditioned random walk. On the unconditioned
problem the best known to us is $r^{5/2}$, see \cite{LPSZ20}. For
the purposes of this paper the exact value of the exponent is not
important.
\begin{proof}
Denote by $M_{n}$ the expectation in the lemma, i.e.
\[
M_{n}\coloneqq\E(\max_{x\in G}\E(|\{t<n:R_{t}\in B(x,r)\}|\,|\,R_{n})).
\]
Examine for one $x$ and one value of $R_{n}$ the quantity $\E(|\{t<n:R_{t}\in B(x,r)\}|\,|\,R_{n})$.
We wish to write it as $A_{1}+A_{2}+A_{3}+A_{4}$, dividing according
to time. For symmetry reasons we have to be slightly careful. Define
$n_{0}=0$, $n_{1}=\lceil n/4\rceil$, $n_{2}=\lceil n/2\rceil$,
$n_{3}=n-\lceil n/4\rceil+1$ and $n_{4}=n+1$ and define
\[
A_{i}\coloneqq A_{i}(x)\coloneqq\E\left[\left|\set{t\in\left[n_{i-1},n_{i}\right):R_{t}\in B(x,r)}\right|\suchthat R_{n}\right].
\]
Now for every 4 non-negative numbers $A$, $B$, $C$ and $D$ we
have
\begin{multline}
A+B+C+D\le\max\{A+B+C,B+C+D\}\\
+\;D\mathbbm{1}\{A+B>0\}+A\mathbbm{1}\{C+D>0\}.\label{eq:ABCD}
\end{multline}
We apply this to $A_{1}(x),\dotsc,A_{4}(x)$ and maximise over $x$.
For brevity we write $A_{123}(x)=A_{1}(x)+A_{2}(x)+A_{3}(x)$ and
similarly for other subsets of $\{1,\dotsc,4\}$. We get
\begin{align}
\max_{x}\{A_{1234}(x)\} & \stackrel{\textrm{\eqref{eq:ABCD}}}{\le}\max_{x}\{\max\{A_{123}(x),A_{234}(x)\}+\nonumber \\
 & \qquad+\;A_{4}(x)\mathbbm{1}\{A_{12}(x)>0\}+A_{1}(x)\mathbbm{1}\{A_{34}(x)>0\}\}\}\nonumber \\
 & \le\max\{\max_{x}\{A_{123}(x)\},\max_{x}\{A_{234}(x)\}+\nonumber \\
 & \qquad+\;\max_{x}\{A_{4}(x)\mathbbm{1}\{A_{12}(x)>0\}\}\nonumber \\
 & \qquad+\max_{x}\{A_{1}(x)\mathbbm{1}\{A_{34}(x)>0\}\}\eqqcolon I+II+III.\label{eq:I II III}
\end{align}
To estimate term $I$ we claim that
\begin{equation}
\max_{x}A_{123}(x)=\max_{x}A_{234}(x).\label{eq:timerev}
\end{equation}
To see \eqref{eq:timerev} note that $A_{123}$ is simply
\[
\E(|\{t\in[0,n_3-1]:R_{t}\in B(x,r)\}|\,|\,R_{n})
\]
and, similarly, $A_{234}$ is
for the interval $[\lceil n/4\rceil,n]$. The time reversal of a random
walk starting from $1$ and conditioned to end at some $y$ is random
walk starting from $y$ and conditioned to end at $1$, so $A_{123}(x)=A_{234}(R_{n}^{-1}x)$. This shows
\eqref{eq:timerev}. 

Fix now one $z$ such that $\Pr(R_{n}=z)>0$ and let $S$ be a random walk of length $n$ conditioned on $S_n=z$. For every $x_0$ We may write
\begin{align*}
\E(|\{t\in [0,n_3-1]:S_t\in B(x_0,r)\}|)
=\E\big(&\E\big(|\{t\in [0,n_3-1]:S_t\in B(x_0,r)\}|\,\big|\,S_{n_3}\big)\big)\\
\le \E\big(\max_x\,&\E\big(|\{t\in [0,n_3-1]:S_t\in B(x,r)\}|\,\big|\,S_{n_3}\big)\big).
\end{align*}
Maximizing over $x_0$ (note that the right-hand side does not depend on $x_0$) and integrating over $z$ gives
\begin{equation}\label{eq:E(I)}
\E(I)\le M_{n_3}.
\end{equation}
This terminates the estimate of $I$.

We move to the estimates of $II$ and $III$. They are essentially
identical, so for brevity we prove only $II$. For every $x$,
\begin{align*}
\lefteqn{{A_{4}(x)\mathbbm{1}\{A_{12}(x)>0\}}}\qquad\qquad\\
 & \le\E\big(|\{(t,s):t\ge n_{3},s<n_{2},R_{t},R_s\in B(x,r)\}|\,\big|\,R_{n}\big).\\
 & \le\sum_{t=n_{3}}^{n_{4}}\sum_{s=0}^{n_{2}}\Pr(d(R_{t},R_{s})\le2r\,\big|\,R_{n}).
\end{align*}
Since the right-hand side does not contain $x$, it also bounds $II$.
Integrating over~$R_{n}$ gives
\begin{equation}
\E(II)\le\sum_{t=n_{3}}^{n_{4}}\sum_{s=0}^{n_{2}}\Pr(d(R_{t},R_{s})\le2r)\le Cn^{2}\max_{m\ge n/4-1}\Pr(d(1,R_{m})\le2r).\label{eq:didnt expect this integration, did you}
\end{equation}
To estimate the right-hand side, we use the results of Lyons, Peres,
Sun and Zheng. By \cite[corollary 2.6]{LPSZ20}, for every integers
$m$ and $r$ and any two $x,y\in G$,
\[
p_{m}(x,y)<C(m^{-10}r^{20}/|B(2r)|+e^{-cm/r^{2}})
\]
(in \cite{LPSZ20} there is an additional parameter $k$, which we
set to $20$). Note that it is at this point that we use that $G$
has superpolynomial growth. Inserting this estimate into \eqref{eq:didnt expect this integration, did you}
gives
\[
\E(II)\le Cn^{-8}r^{20}+Cn^{2}|B(2r)|e^{-cn/r^{2}}.
\]
If $n>C'r^{3}$ for some $C'$ sufficient large then $\E(II)\le C/n$.

We can now wrap up the lemma. If $n\le C'r^{3}$ then the claim of
the lemma holds trivially. For any $n>C'r^{3}$ we have
\[
M_{n}\le\E(\max_{x}A_{1234}(x))\stackrel{\textrm{\eqref{eq:I II III}}}{\le}\E(I)+\E(II)+\E(III)\stackrel{(*)}{\le}M_{n_{3}}+C/n
\]
where in $(*)$ we used \eqref{eq:E(I)} and the estimate of $II$
above (which, as explained, applies verbatim to $III$). Recall that
$n_{3}=n-\lceil n/4\rceil+1$. The lemma now follows by a straightforward
induction on $n$.
\end{proof}
\begin{lem}
\label{lem:not just superpoly}The conclusion of \Lref{lem:superpoly}
holds for any transient group.
\end{lem}

\begin{proof}
Since \Lref{lem:superpoly} proves this for $G$ with superpolynomial
growth, we need only cover the case that $G$ has polynomial growth.
However, for $G$ with growth $d$ we have the estimates
\[
p_{n}(1,y)\le Cn^{-d/2}e^{-c|y|^{2}/n}.
\]
This is well known, but for completeness let us explain this estimate. By Gromov's theorem \cite{Gro81,Kle10} the group $G$ is virtually nilpotent. It is not difficult to see that any nilpotent group is virtually torsion free \cite[chapter II]{ragbook}. By Mal\textquotesingle cev's theorem \cite[chapter II]{ragbook}, a nilpotent torsion-free group is embedded as a cocompact lattice in a connected nilpotent Lie group. By the results of Guivarc'h \cite{gui73} there is an integer $d$ such that the Lie group has volume growth $d$. Since a Lie group is quasi-isometric to a cocompact lattice in it, and since our cocompact lattice is quasi-isometric to $G$, and since growth is preserved by quasi-isometries, we get that $G$ satisfies that $cr^d\le |B(1,r)|\le Cr^d$.  Finally the results of Hebisch and Saloff-Coste \cite[theorems 2.1 and 4.1]{hsc93} give the estimate above. 

The probability to be at $z$ at time $t$ and at $y$ at time $n>t$ can
then be bounded by
\[
Ct^{-d/2}(n-t)^{-d/2}e^{-c|z|^{2}/t-c|y-z|^{2}/(n-t)}.
\]
Summing this estimate over $t\le n$ gives
\[
C\left(\min\{|z|,|y-z|\}\right)^{2-d}n^{-d/2}e^{-c|y|^{2}/n}
\]
(using $d\ge3$, which follows from our assumption that the group
is transient). Summing over $z\in B(x,r)$ gives $r^{2}n^{-d/2}e^{-c|y|^{2}/n}$,
independently of $x$ . Thus
\[
\max_{x}\E\big(|\{t<n:R_{t}\in B(x,r)\}|\,\big|\,R_{n}=y\big)\le\frac{Cr^{2}n^{-d/2}e^{-c|y|^{2}/n}}{\Pr(R_{n}=y)}.
\]
Taking expectation over $R_{n}$ gives
\[
\E\big(\max_{x}\E\big(|\{t<n:R_{t}\in B(x,r)\}|\,\big|\,R_{n}\big)\big)\le Cr^{2}n^{-d/2}\sum_{y}e^{-c|y|^{2}/n}\le Cr^{2},
\]
where in the last inequality we used the fact that $|B(1,r)|\le Cr^d$. This
proves the lemma.
\end{proof}
\begin{lem}
\label{lem:LPSZ for paths}Let $G$ be a finitely generated group,
let $S$ be a set of generators, let~$w$ be a word and let $\gamma$
be the path corresponding to $w$ and random walks of length~$n$ with repsect to the generating set $S$. Let $x\in G$ and let
$r\in\mathbb{N}$. Then
\[
\E(|\gamma\cap B(x,r)|)\le Cr^{3}.
\]
(the constant $C$ may depend on $w$).
\end{lem}

\begin{proof}
The lemma holds trivially for $G$ recurrent, since the only recurrent groups are those having volume growth 1 and 2. Assume therefore that $G$ is transient.
Let $l$ be the length of $w$ and let $a_{i}$ be its letters,
so $w=(a_{1}\dotsc a_{l})$. Let $w_{i}$ be the word comprising
of the first $i$ letters of $w$, for $i\in\{0,\dotsc,l-1\}$.
Let $R^{(1)},\dotsc,R^{(d)}$ be the random walks constructing $\gamma$.
Condition on $R_{n}^{(1)},\dotsc,R_{n}^{(d)}$. Then
\begin{align*}
\lefteqn{{\E\big(|\gamma\cap B(x,r)|\,\big|\,R_{n}^{(1)},\dotsc,R_{n}^{(d)}\big)}}\qquad\qquad\\
 & \le\sum_{i=0}^{l-1}\E\big(|w_{i}(R_{n}^{(1)},\dotsc,R_{n}^{(d)})R^{(a_{i+1})}\cap B(x,r)|\,\big|\,R_n^{(1)},\dotsc,R_n^{(d)}\big)\\
 & =\sum_{i=0}^{l-1}\E\big(|R^{(a_{i+1})}\cap B(w_{i}(R_{n}^{(1)},\dotsc,R_{n}^{(d)})^{-1}x,r)|\,\big|\,R_n^{(1)},\dotsc,R_n^{(d)}\big)\\
 & \le\sum_{i=0}^{l-1}\max_{y\in G}\E\big(|R^{(a_{i+1})}\cap B(y,r)|\,\big|\,R_n^{(1)},\dotsc,R_n^{(d)}\big)\\
 & =\sum_{i=0}^{l-1}\max_{y\in G}\E\big(|R^{(a_{i+1})}\cap B(y,r)|\,\big|\,R_n^{(a_{i+1})}\big).
\end{align*}
Taking expectation over $R_{n}^{(1)},\dotsc,R_{n}^{(d)}$ we may estimate
each term by \Lref{lem:not just superpoly} and get
\[
\E(|\gamma\cap B(x,r)|)\le Clr^{3},
\]
as needed.
\end{proof}

\section{Commutators of a power law with itself}\label{sec:self-comm-power}

\begin{defn}
Let $w$ be a word on $k$ letters, let $G$ be finitely generated
group and let $S$ be a finite set of generators. Let $\gamma_{n}$
be the path corresponding to the word~$w$ and $R^{(1)},\dotsc,R^{(d)}$,
where $R^{(1)},\dotsc,R^{(d)}$ are independent random walks on~$G$ of
length~$n$.

We say that $w$ is \textbf{non-Peano} on $G$, $S$ if
\[
\max_{n}\Pr(g\in\gamma_{n})\le C\left|g\right|^{-4},
\]
for some $C$ independent of $n$ and $g$ ($C$ may depend on $G$,
$S$ and $w$).
\end{defn}

We call this property non-Peano because such a path is far from visiting all of~$G$.
Clearly $G$ must have at least $8$ dimensional volume growth to be
non-Peano for any~$w$. But we do not believe this is enough.
Presumably it is possible to construct artificial examples of a $w$
and a ``large'' group $G$ such that $w$ is Peano on~$G$, but
we decided not to pursue this direction here.
\begin{lem}
\label{lem:non-peano is enough}Let $G$ be a finitely generated group
and let $S$ be a finite set of generators. Let $H$ be a $2$-generated group. Let $w_{1}$ and $w_{2}$ be two laws on $d_1$ and~$d_2$ letters respectively such that
$$
\liminf_{n\to\infty}\,\Pr(w_i(R_n^{(1)},\dotsc,R_n^{(d_i)})=1)=p_i\qquad i\in\{1,2\}
$$ for some $p_1$ and $p_2$ (note that we need here $\liminf$ rather than our usual $\limsup$). Assume also that $w_1$ and $w_2$ are non-Peano
on $G$, $S$. Then for every $\varepsilon>0$ there are generators
$S'$ for $H\wr G$ such that the law 
$$[w_{1},w_{2}]\coloneqq[w_1(R_n^{(1)},\dotsc,R_n^{(d_1)}),w_2(R_n^{(d_1+1)},\dotsc,R_n^{(d_1+d_2)})]$$
is satisfied on $H\wr G$ with probability at least $p_1p_2-\varepsilon$.
\end{lem}

\begin{proof}
Let $u\in G$ be some element, sufficiently large, to be chosen later
depending only on $\varepsilon$. Let $a$ and $b$ be the generators
of $H$. We define the modified switch-move-switch generators
for $H\wr G$ by
$$
S'=\{s_1ms_2 : s_i\in \{a^{\pm 1}\delta_1,b^{\pm 1}\delta_u\}, m\in S\}
$$

Let us estimate the probability that $[w_{1},w_{2}]$ is satisfied under random walk. Write $w_i(\vec{R}^i)=(L_i,g_i)$ for $i=1,2$, where $g_i\in G$ and $L_i\in\bigoplus_{g\in G} H$. We use again \Fref{f:commutators}: $[w_1(\vec{R}^1),w_2(\vec{R}^2)]=1$ in $H\wr G$ if $g_1=g_2=1$ and $[L_1(g),L_2(g)]=1$ for all $g\in G$.

Write $D$ for the event that $g_1=g_2=1$. Since $G$ satisfies $w_1$ and $w_2$ with probabilities $p_1$ and $p_2$ respectively, and the events $g_1=1$ and $g_2=1$ are independent (recall that $w_1$ and $w_2$ are laws on independent random walks), $\Pr\left(D\right)\ge p_1p_2$.

For any $g\in G$, let $A(g)=\set{[L_1(g),L_2(g)]=1}$. Then
\begin{align}
    \Pr\left([w_1(\vec{R}^1),w_2(\vec{R}^2)]=1\right) & \ge \Pr\bigg(D\cap\bigcap_{g\in G}A(g)\bigg)\ge \Pr\left(D\right)-\Pr\bigg(\bigcup_{g\in G}A(g)^c\bigg) \ge \nonumber\\
    & \ge \Pr\left(D\right) - \sum_{g\in G}\Pr\left(A(g)^c\right).\label{eq:prob-estimate-peano-com}
\end{align}

For ease of notation, write $\gamma_i=\gamma(w_i;\vec{R}^i)$ to be the path defined by $w_i$ under $\vec{R}$. Note that if $L_i(g)\notin\sg{b}$ we must have $g\in\gamma_i$, and if $L_i(g)\notin\sg{a}$ then $gu^{-1}\in\gamma_i$.

To estimate $\sum_{g\in G}\Pr\left(A(g)^c\right)$, first note that if $[L_1(g),L_2(g)]\neq 1$, then at least one of them used an $a$ lamp, while the other used a $b$ lamp (otherwise both lie in~$\sg{a}$ or in $\sg{b}$, which are abelian subgroups). In particular, for any $g$ we have
\begin{align*}
    \Pr\left(A(g)^c\right)&\le \Pr\left(L_1(g)\notin\sg{a},L_2(g)\notin\sg{b}\right) + \Pr\left(L_1(g)\notin\sg{b},L_2(g)\notin\sg{a}\right)=\\
    &=\Pr\left(gu^{-1}\in\gamma_1,g\in\gamma_2\right) + \Pr\left(g\in\gamma_1,gu^{-1}\in\gamma_2\right).
\end{align*}
Since for every $g$ the paths $\gamma_1,\gamma_2$ are independent (recall that the notation $[w_{1},w_{2}]$ means that the two laws are used with different, independent random walks), we have
\[
  \Pr\left(A(g)^c\right)\le \Pr\left(gu^{-1}\in\gamma_1\right)\Pr\left(g\in\gamma_2\right) + \Pr\left(g\in\gamma_1\right)\Pr\left(gu^{-1}\in\gamma_2\right),
\]
and thus
\[
  \sum_{g\in G}\Pr\left(A(g)^c\right)\le \sum_{g\in G}\Pr\left(gu^{-1}\in\gamma_1\right)\Pr\left(g\in\gamma_2\right) + \sum_{g\in G}\Pr\left(g\in\gamma_1\right)\Pr\left(gu^{-1}\in\gamma_2\right).
\]

Write $G=B_{L}(u,\left|u\right|/2)\cupdot V$
where $B_{L}$ is the left ball around $u$, i.e.\ $B_L(u,|u|/2)=\set{gu:\left|g\right|<\left|u\right|/2}$, and where $V$ is the remainder. We will estimate each part independently.

We start with $B_L(u,\left|u\right|/2)$. For any $g\in B_L(u,\left|u\right|/2)$ we have $d(g,1)>\left|u\right|/2$. Since $w_i$ is non-Peano, we have
\[
  \Pr\left(g\in\gamma_i\right)\le C\left|u\right|^{-4}.
\]
Next, by \Lref{lem:LPSZ for paths},
\[
  \sum_{g\in B_L(u,\left|u\right|/2)}\Pr\left(gu^{-1}\in\gamma_i\right)= \E\left[\left|(B(1,\left|u\right|/2))\cap\gamma_i\right|\right]\le C\left|u\right|^3,
\]
where the equality uses that when the center is the identity, a left ball is the same as a right ball. We can therefore conclude that
\[
  \sum_{g\in B_L(u,\left|u\right|/2)}\Pr\left(A(g)^c\right)\le C\left|u\right|^{-1}.
\]

To handle $V$ we write $V=\bigcup_{r=0}^{\infty}V_r$, where $V_0=V\cap B(1,\left|u\right|)$, and $V_r=V\cap\left(B(1,2^r\left|u\right|)\setminus B(1,2^{r-1}\left|u\right|)\right)$ for any $r\ge 1$. For any $g\in V_r$, since $w_i$ is non-Peano,
\[
  \Pr\left(gu^{-1}\in\gamma_i\right)\le C\left(2^r\left|u\right|\right)^{-4},
\]
where we estimate that $|gu^{-1}|>c2^r\left|u\right|$ because $\left|g\right|>2^{r-1}\left|u\right|$ for $r\ge 1$, and $g\notin B_L(u,\left|u\right|/2)$ for $r=0$. In addition, again with \Lref{lem:LPSZ for paths},
\[
  \sum_{g\in V_r}\Pr\left(g\in\gamma_i\right)\le \E\left[\left|B(1,2^r\left|u\right|)\cap\gamma_i\right|\right]\le C\left(2^r\left|u\right|\right)^3.
\]
This proves that we may again conclude
\[
  \sum_{g\in V_r}\Pr\left(A(g)^c\right)\le C\left(2^r\left|u\right|\right)^{-1}.
\]

In conclusion,
\begin{align*}
  \sum_{g\in G}\Pr\left(A(g)^c\right) &= \sum_{g\in B_L(u,\left|u\right|/2)}\Pr\left(A(g)^c\right) + \sum_{r=0}^{\infty}\sum_{g\in V_r}\Pr\left(A(g)^c\right)\le\\
  & \le 2C\left|u\right|^{-1}+\sum_{r=0}^{\infty}C(2^r\left|u\right|)^{-1}\le C\left|u\right|^{-1}.
\end{align*}

We may choose $\left|u\right|$ sufficiently large such that $C\left|u\right|^{-1}\leq\eps$, and thus we have that $\Pr\left(\bigcap A(g)\right)\ge 1-\eps$. Therefore
\[
\Pr\left([w_1(\vec{R}^1),w_2(\vec{R}^2)]=1\right) \ge \Pr\left(D\cap \bigcap A(g)\right) \ge \Pr\left(D\right)+\Pr\left(\bigcap A(g)\right)-1\ge p_1p_2-\eps
\]
as required.
\end{proof}

Before moving on to Burnside groups, we need two simple lemmas. One
on paths in lamplighter groups, and another on linear algebra.
\begin{lem}
\label{lem:spanning}Let $B$ be a $d$-generated group and let $G=\mathbb{Z}/l\mathbb{Z}\wr B$
with the switch-move-switch generators, with respect to lazy simple
random walk on $\mathbb{Z}/l\mathbb{Z}$ and the natural walk on $B$.
Let $\gamma$ be a geodesic path of length $n$ in $G$ and let $\beta$ be
its projection to~$B$. Then
\[
|\{\beta_{0},\dotsc,\beta_{n}\}|\ge\frac{1}{l/2+2}|\gamma_{n}|.
\]
\end{lem}

(Recall that a geodesic path is the shortest path between its endpoints).

\begin{proof}
We will show that for any path $\gamma$ there exists a path $\gamma'$
with length $n'\le n$ with the same projection to $B$ and same endpoint,
such that for all $b$, the projection $\beta'$ of $\gamma'$ visits
$b$ no more than $l/2+2$ times. This will show the lemma by
contradiction since if not, then
\[
n'\le|\{\beta'_{0},\dotsc,\beta_{n}'\}|(l/2+2)=|\{\beta_{0},\dotsc,\beta_{n}\}|(l/2+2)<|\gamma_{n}|=|\gamma_{n}'|
\]
which contradicts the definition of $|\gamma_{n}|$ as the length
of the shortest path.

To construct $\gamma'$ we first modify $\gamma$ so that at every
visit to every $b$ it does a geodesic in $\mathbb{Z}/l\mathbb{Z}$
towards $\gamma_{n}(b)$ and when it reaches that value it no longer
moves (recall that the walk on $\mathbb{Z}/l\mathbb{Z}$ is lazy).
Since at every visit to $b$ it can move the value by $2$ (once when
entering and once when exiting), and since the diameter of~$\mathbb{Z}/l\mathbb{Z}$
is $\lfloor l/2\rfloor$, we get that after $\lfloor l/4\rfloor+1$
visits to $b$, it has already taken the value~$\gamma_{n}(b)$.

Next, since the set of sites visited by $\beta$ is connected, one
may find a spanning tree. A simple induction shows that any tree can
be traversed so that each vertex is visited at least once and the overall number of is at most twice the size of the tree (e.g. doing Depth first search). Therefore to get the the desired position and configuration, we
repeat this traversal $\lfloor l/4\rfloor$ times to set all the lamps to the right configuration, then take a geodesic to the desired endpoint. 
\end{proof}
\begin{lem}
\label{lem:kugelager}Let $A$ be an $m\times m'$ matrix with $\mathbb{Z}/l\mathbb{Z}$
entries, such that no row of $A$ is identically zero. Assume that
for some $k$ we have that every row of $A$ contains no more than
$k$ non-zero entries, and similarly for all columns of $A$. Let $X_{1},\dotsc,X_{m'}$
be independent (not necessarily identically distributed) random $\mathbb{Z}/l\mathbb{Z}$
variables with the property that $X_{i}$ takes no specific value
with probability bigger than $1-\varepsilon$ (with~$\varepsilon$
independent of $i$). Then for any $v\in(\mathbb{Z}/l\mathbb{Z})^{m}$
\[
\Pr(AX=v)\le(1-\varepsilon)^{\lceil m/k^{2}\rceil}.
\]
\end{lem}

\begin{proof}
Consider the first row of $A$ (think about rows as equations). Condition
on the value of all $X_{i}$ that appear in that row except one, we
see that in order for this equation to be satisfied that last variable
must take a specific value, an event that has probability at most
$1-\varepsilon$.

Now, since the first row contains at most $k$ non-zero entry, and
since for any such entry that are at most $k$ rows containing it,
we see that at most $k^{2}$ rows of~$A$ have any joint variable
with the first row. Discarding these rows, the next one is independent
of the first. Since it, too, has probability $1-\varepsilon$ to be
satisfied, the probability that both are satisfied is at most $(1-\varepsilon)^{2}$.
Continuing this way, the lemma is proved.
\end{proof}
\begin{lem}
\label{lem:burnside-laplighter-non-peano}Let $q$ be a sufficiently
large composite number. Then there exists a group $G$ satisfying
$x^{q}=1$ for all $x\in G$ and a finite set of generators $S$ of
$G$ such that $x^{q}$ is non-Peano for $G$, $S$.
\end{lem}

\begin{proof}
For any $k_{0}$ there exists $q_{0}$ such that if $q>q_{0}$ is
composite then it is possible to write $q=kl$ with $k>k_{0}$ and
with $l>1$. By \cite{I94} for any $k\ge2^{57}$ the free Burnside
group $B$ of exponent $k$ on $2$ generators is infinite (the formulation
in \cite{I94} is ``$k$ larger than $2^{48}$ and either odd or divisible
by $2^{9}$'' but this is the same. Indeed, if the Burnside group of
a given exponent $k$ is infinite then so is the Burnside group of
every product of $k$). We take $q_{0}$ corresponding to $k_{0}=2^{57}$
in the above. Consider the group $G=\mathbb{Z}/l\mathbb{Z}\wr B$.
Then it is easy to see that $G$ is of exponent $q$. Indeed, for
any $g=(\omega,b)\in G$, $g^{k}=(\xi,b^{k})=(\xi,1)$ for some $\xi:B\to\mathbb{Z}/l\mathbb{Z}$
and then
\[
g^{q}=\left(g^{k}\right)^{l}=(\xi,1)^{l}=(l\xi,1)=(\vec{0},1)
\]
which is the identity in $G$. We take this $G$ with the standard
switch-move-switch generators with respect to the lazy random walk
on $\mathbb{Z}/l\mathbb{Z}$ and the standard walk on $B$.

Thus we need only show that $x^{q}$ is non-Peano, i.e.\ that $\Pr(g\in\gamma)\le C|g|^{-4}$
for all $n$, where $\gamma$ is the path of $R^{q}$ and $R$ is
a random walk of length $n$. To see this condition on the path of
$R$ in $B$ and call this path $\beta$ (we apologise for the confusing
use of the word ``path'' for both the path in $B$ and in $G$ ---
hopefully, if the reader keeps in mind that the path in $B$ is denoted
by $\beta$ and the path in $G$ is denoted by~$\gamma$ not too much
confusion will arise). Denote $g\eqqcolon(\omega,b)$ and consider
one $i$ such that $\beta_{i}=b$. By \Lref{lem:spanning} we
know that the set $M\coloneqq\{\beta_{0},\dotsc,\beta_{i}\}$ must
have size at least $c|g|$ (here and below constants may depend on
$q$, $k$ and $l$), otherwise it is impossible for $\gamma_{i}$
to be $g$ and we may discard such $i$.

Write $i=an+r$ for $0\leq r<n$ (and $a\leq q$), and write $R_n=R_rR'_{n-r}$, where $R_r=(\xi,g_0)$ and $R'_{n-r}=(\xi',g_0')$ are independent random walks with $r$ and $n-r$ steps respectively. As $R_r$ and $R'_{n-r}$ are independent, and after conditioning on the path $\beta$, the random variables $\set{\xi(g)}_{g\in G}$ and $\set{\xi'(g)}_{g\in G}$ are independent and non-degenerate.

As $\gamma_i=(R_rR'_{n-r})^aR_r$, for any $x\in B$ which is visited by the random walk $R_n$ by time $n$, $\gamma_i(x)$ is a linear combination of the variables $\set{\xi(g)}_{g\in G}$ and $\set{\xi'(g)}_{g\in G}$. Each $\gamma_i(x)$ is a linear combination of at most $2q$ such variables, and each such variable appears in at most $2q$ combinations. Thus
we are exactly in the situation of \Lref{lem:kugelager} and
we may conclude that
\[
\Pr(\gamma_{i}=g\,|\,\beta)\le(1-c)^{\lceil|M|/4q^{2}\rceil}.
\]
Together with the previous observation we have, still for $i$ such
that $\beta_{i}=b$, that $\Pr(\gamma_{i}=g\,|\,\beta)\le\exp(-c|g|)$.

Finally, to estimate the number of relevant $i$ we note that $B$
is transient and hence using \Lref{lem:not just superpoly} with
$r=1$ shows that, conditioning only on $\beta_{n}$ and not on the
whole of $\beta$, each piece of $\beta$ visits $b$ in expectation at most a constant number
of times. Integrating this conditioning shows that
\[
\E(|\{i:\beta_{i}=b\}|)\le Cq.
\]
The lemma is then finished by the following calculation
\[
\Pr(g\in\gamma)\le\sum_{i=0}^{qn}\E(\Pr(\gamma_{i}=g\,|\,\beta))\le\sum_{i=0}^{qn}\Pr(\beta_{i}=b)\exp(-c|g|)\le Cq\exp(-c|g|),
\]
which is far stronger than the $8$-dimensional estimate we needed
to prove.
\end{proof}
\begin{rem*}
The idea that a lamplighter with a Burnside base group is also a Burnside
group has been used before, see the appendix of \cite{O18} and a
reference within.
\end{rem*}
\begin{thm}
For $q$ sufficiently large and composite, the law $[x^{q},y^{q}]$
has no gap or positivity for infinite groups.
\end{thm}

This follows immediately from \Lref{lem:non-peano is enough}, by considering the wreath product $\Gamma=\F_2\wr G$ for the group $G$ from \Lref{lem:burnside-laplighter-non-peano}, using the fact that $\F_{2}\wr G$
does not satisfy any law (even virtually), due to its free subgroups.

\section{Uniform measures on balls}\label{sec:uniform-measures-balls}

In this section we deal with the measures induced by taking uniform measures on balls. I.e., we choose a finite symmetric generating set $S$ containing $1$, and let~$\mu_n$ denote the uniform measure on~$S^n$. We show that in this case that many laws do not have a strong gap.

\begin{thm}\label{thm:uniform-balls}
Let $w$ be a non-trivial law such that there is some group of exponential growth satisfying $w$. Then $w$ has no strong gap with respect to measures induced by uniform measures on balls.
\end{thm}
\begin{proof}
Write $w=w(x_1,\dots,x_d)$, and let $G$ be a group of exponential growth satisfying $w$. Take a finite symmetric generating set $S$ of $G$ containing $1$, and let $\rho=\lim_{n\to\infty}\sqrt[n]{|S^n|}>1$ be the growth exponent of $G$ with respect to $S$. Also, let $H$ be a finite group which does not satisfy $w$ (it is easy to see that any non-trivial law is not satisfied by a sufficiently large symmetric group).

Consider the group $\Gamma=G\times H$, with generating set $T=(S\times \set{1})\cup(\set{1}\times H)$, and let $\mu_n$ denote the uniform measure on $T^n$. Then for any $n\ge 1$ and $h\in H$ we have
$$\left|\set{(g,h)\in T^n\mid g\in G}\right|=\begin{cases}|S^{n-1}|,&h\neq 1\\|S^n|,&h=1.\end{cases}$$
Note that $w((g_1,h_1),\dots,(g_d,h_d))=1$ if and only if $w(h_1,\dots,h_d)=1$ (because~$G$ satisfies $w$), and thus $w$ holds if $h_1=\cdots=h_d=1$. Therefore, if $(g_i,h_i)$ are distributed uniformly in a ball of radius $n$ then 
\begin{align*}
    \Pr
    (w((g_1,h_1),\dots,(g_d,h_d))=1)&\ge \Pr
    (h_1=\cdots=h_d=1)=\\
    &=\frac{|S^n|^d}{\left((|H|-1)|S^{n-1}|+|S^n|\right)^d}=\\
    &=\frac{1}{\left((|H|-1)\frac{|S^{n-1}|}{|S^n|}+1\right)^d},
\end{align*}
which shows that
$$\limsup_{n\to\infty}\Pr
(w((g_1,h_1),\dots,(g_d,h_d))=1)\ge\frac{1}{\left((|H|-1)\rho^{-1}+1\right)^d}.$$
By replacing $S$ with $S^m$ for large enough $m$ we may take $\rho\to\infty$, and the above probability converges to $1$. However, it is clear that $\Gamma$ does not satisfy $w$, as required.
\end{proof}

\begin{rem*}
    Any law satisfying the assumption of \Tref{thm:uniform-balls} does not satisfy a strong positivity result as well (with respect to uniform measures on balls).
    
    Indeed, let $w$ be such a law, and let $G$ be a group of exponential growth satisfying~$w$. Fix $\eps>0$. Following the proof of the theorem, for any $k$ we may find a generating set of $G\times\Sym_k$ such that $\Pr\left(w=1\on G\times\Sym_k\right)\ge\eps$ (with respect to balls in the appropriate Cayley graph of $G\times\Sym_k$). However, $G\times\Sym_k$ cannot have a subgroup satisfying $w$ for which the index is bounded by a function of $\eps$, since this is not true for $\Sym_k$ as $k$ tends to $\infty$.
\end{rem*}

\begin{exmpl}
We give two examples of law satisfying the assumption of \Tref{thm:uniform-balls}:
\begin{enumerate}
    \item The law $x^n=1$ for large enough $n$ such that the Burnside group $B(d,n)$ is infinite (for some~$d$).
    \item The metabelian law $[[x,y],[z,w]]=1$ (for which we can take $G=\Z/2\Z\wr\Z$).
\end{enumerate}
\end{exmpl}

\section{Open questions}

We conclude this paper with several open questions. For finite groups (or for residually finite groups), the existence of a gap (\Qref{ques:gap-fin}) and the positivity (\Qref{ques:pos-fin}) are still open. 

A particular case of interest is for power laws $x^m=1$. When $G$ is finite, it is known that there exists a gap for $\Pr\left(x^m=1\on G\right)$ depending on $m$ and on the number of generators of $G$. However, the dependency on the number of generators is still, as far as we know, unknown:

\begin{ques}
    Does the law $x^m=1$ satisfy a gap result for finite groups, with no dependency on the number of generators?
\end{ques}

This question is open for both the finite case and the infinite case with respect to random walks (though we believe that the two cases might be very different). Note that for $m=3$ and $4$ Laffey gave a positive answer \cite{laf,laf3}. Further, for $m=p$ a prime number, Laffey shows that any non-$p$-groups must have $\Pr(x^p=1)\le p/(p+1)$ \cite{laf2}. For $p$-groups the situation is more complicated, and Khukhro gave an example of a 5-group where $1>\Pr(x^5=1)\ge 24/25$ \cite{kh81}. See \cite{hvl09} for a survey of this direction.

As a generalization of several results provided in this paper, we raise the following question:
\begin{ques}
    Let $w_1=w_1(x_1,\dots,x_d)\in\F_d$ and $w_2=w_2(y_1,\dots,y_{d'})\in\F_d$ be words. Does the law $[w_1,w_2]$ satisfy a gap or positivity result?
\end{ques}

Recall that our results show that when $w_1,w_2$ are balanced, or when $w_1,w_2$ are the same word with disjoint variables, then $[w_1,w_2]$ does not satisfy a gap or a positivity result; also, if $w_1$ is a power law, then $[w_1,w_2]$ does not satisfy a positivity result.\medskip

Our next questions are somewhat ``extreme'' cases for the gap and positivity of laws. We showed that a law can be satisfied with probability arbitrarily close to $1$ on a group, while the group does not satisfy the law (or any law) virtually. One may ask for the following:

\begin{ques}
    Let $w\in\F_d$ be a word, let $G$ be a finitely generated group generated by $S$, and let $R_n^{(1)},\dots,R_n^{(d)}$ be independent random walks on $G$ with respect to $S$. If $\lim_{n\to\infty}\Pr(w(R_n^{(1)},\dots,R_n^{(d)})=1)=1$, is it true that $w=1$ holds on $G$?
\end{ques}

Regarding positivity,
\begin{ques}
    Can the probability of satisfying a law be $0$ for one generating set but positive for another?
\end{ques}

Finally, we raise another question, a bit more technical in nature:
\begin{ques}
    Can the $\limsup\neq\liminf$ of the probability to satisfy $w$?
\end{ques}
This can be arranged, for example, if the random walk is not lazy (consider for example the infinite Dihedral group $D_{\infty}=\sg{a,b\suchthat a^2=b^2=1}$, with respect to the generating set $S=\set{a,b}$, and the law $x^2=1$). However, when the random walk is lazy, it is unclear whether the probabilities converge or not. We conjecture that, in fact, an example where the probabilities do not converge can be constructed.

\bibliographystyle{plain}
\bibliography{refs}

\end{document}